\documentclass[times,review]{elsarticle}
\usepackage{framed,multirow}

\usepackage[utf8]{inputenc}

\usepackage{graphicx}
\usepackage[version=4]{mhchem}
\usepackage{longtable,tabularx}
\usepackage{amsmath,amssymb,amsthm}

\usepackage{csquotes}

\usepackage{geometry}
\usepackage{float}

\usepackage{url}
\usepackage{pifont}
\usepackage{empheq}
\usepackage{bm}
\usepackage{soul}

\usepackage{scalerel}

\usepackage{amsthm}

\usepackage{csquotes}

\usepackage{booktabs}
\usepackage{geometry}
\usepackage{float}

\usepackage{url}
\usepackage{pifont}

\usepackage{empheq}
\usepackage{bm}
\usepackage[english]{babel} 
\addto\captionsenglish{\renewcommand{\appendixname}{Appendix }}
\usepackage{amsmath,amssymb}
\usepackage{epstopdf}
\usepackage{relsize}
\usepackage{stmaryrd} 
\usepackage{subcaption}
\usepackage{multirow}
\usepackage{array,makecell}
\usepackage[squaren, Gray, cdot]{SIunits}

\usepackage{import}
\usepackage{diagbox}

\usepackage{csquotes}

\usepackage{xspace}

\usepackage{scalerel}

\addto\captionsenglish{\renewcommand{\appendixname}{}}

\makeatletter
\def\ps@pprintTitle{%
 \let\@oddhead\@empty
 \let\@evenhead\@empty
 \def\@oddfoot{}%
 \let\@evenfoot\@oddfoot}

\usepackage{lettrine}

\newtheorem{thm}{Theorem}[section]
\newtheorem{lemma}{Lemma}[section]

\usepackage{lineno}
\newcommand*\patchAmsMathEnvironmentForLineno[1]{%
\expandafter\let\csname old#1\expandafter\endcsname\csname #1\endcsname
\expandafter\let\csname oldend#1\expandafter\endcsname\csname end#1\endcsname
\renewenvironment{#1}%
{\linenomath\csname old#1\endcsname}%
{\csname oldend#1\endcsname\endlinenomath}}%
\newcommand*\patchBothAmsMathEnvironmentsForLineno[1]{%
\patchAmsMathEnvironmentForLineno{#1}%
\patchAmsMathEnvironmentForLineno{#1*}}%
\AtBeginDocument{%
\patchBothAmsMathEnvironmentsForLineno{equation}%
\patchBothAmsMathEnvironmentsForLineno{align}%
\patchBothAmsMathEnvironmentsForLineno{flalign}%
\patchBothAmsMathEnvironmentsForLineno{alignat}%
\patchBothAmsMathEnvironmentsForLineno{gather}%
\patchBothAmsMathEnvironmentsForLineno{multline}%
}

\usepackage{hyperref}

\usepackage{xcolor}

\makeatletter
\def\ps@pprintTitle{%
 \let\@oddhead\@empty
 \let\@evenhead\@empty
 \def\@oddfoot{}%
 \let\@evenfoot\@oddfoot}

\definecolor{newcolor}{rgb}{.8,.349,.1}

\begin{document}

\begin{frontmatter}

\title{Discretely Nonlinearly Stable Weight-Adjusted Flux Reconstruction High-Order Method for Compressible Flows on Curvilinear Grids}

\author[1]{Alexander {Cicchino}\corref{cor1}\fnref{CICC}}
\cortext[cor1]{Corresponding author. 
}
  \ead{alexander.cicchino@mail.mcgill.ca}
  \fntext[CICC]{Ph.D. Candidate}

\author[1]{Siva {Nadarajah}\fnref{Nadarajah}}
\fntext[Nadarajah]{Professor}
\ead{siva.nadarajah@mcgill.ca}

\address[1]{Department of Mechanical Engineering, McGill University, 
Montreal, QC, H3A 0C3, Canada}






\begin{abstract}
To achieve genuine predictive capability, an algorithm must consistently deliver accurate results over prolonged temporal integration periods, avoiding the unwarranted growth of aliasing errors that compromise the discrete solution. Provable nonlinear stability bounds the discrete approximation and ensures that the discretization does not diverge. Nonlinear stability is accomplished by satisfying a secondary conservation law, namely for compressible flows; the second law of thermodynamics. For high-order methods, discrete nonlinear stability and entropy stability, have been successfully implemented for discontinuous Galerkin (DG) and residual distribution schemes, where the stability proofs depend on properties of L2-norms. In this paper, nonlinearly stable flux reconstruction (NSFR) schemes are developed for three-dimensional compressible flow in curvilinear coordinates. NSFR is derived by merging the energy stable flux reconstruction (ESFR) framework with entropy stable DG schemes. NSFR is demonstrated to use larger time-steps than DG due to the ESFR correction functions, while preserving discrete nonlinear stability. NSFR differs from ESFR schemes in the literature since it incorporates the FR correction functions on the volume terms through the use of a modified mass matrix. We also prove that discrete kinetic energy stability cannot be preserved to machine precision for quadrature rules where the surface quadrature is not a subset of the volume quadrature. This result stems from the inverse mapping from the kinetic energy variables to the conservative variables not existing for the kinetic energy projected variables. This paper also presents the NSFR modified mass matrix in a weight-adjusted form. This form reduces the computational cost in curvilinear coordinates because the dense matrix inversion is approximated by a pre-computed projection operator and the inverse of a diagonal matrix on-the-fly and exploits the tensor product basis functions to utilize sum-factorization. The nonlinear stability properties of the scheme are verified on a nonsymmetric curvilinear grid for the inviscid Taylor-Green vortex problem and the correct orders of convergence were obtained on a curvilinear mesh for a manufactured solution. Lastly, we perform a computational cost comparison between conservative DG, overintegrated DG, and our proposed entropy conserving NSFR scheme, and find that our proposed entropy conserving NSFR scheme is computationally competitive with the conservative DG scheme.
\end{abstract}
\end{frontmatter}

\section{Introduction}\label{sec: Introduction}

High-order methods such as discontinuous Galerkin (DG) and
flux reconstruction (FR), result in efficient computations via high solution accuracy and dense computational kernels, making them an attractive approach for the exascale concurrency on current and next generation hardware. Generally, high-order methods are known to be more efficient than low-order methods for linear hyperbolic time-dependent problems (e.g., see~\cite{kreiss1972comparison,swartz1974relative}). However, despite vigorous efforts by the research community, their application to real world complex problems governed by nonlinear partial differential equations (PDEs) has been limited due to a lack of robustness. 

The DG method, proposed by Reed and Hill~\cite{reed1973triangular}, combines both the key properties of finite volume and finite element schemes. As explained in the book of Hesthaven and Warburton~\cite{NDG}, the high-order scheme provides stability through a numerical flux function, and utilizes high-order shape functions to represent the solution. 
Another variation of the DG method is FR, initially proposed by H.T. Huynh~\cite{huynh_flux_2007} and later presented as a class of energy stable flux reconstruction (ESFR) schemes~\cite{vincent_new_2011,jameson_proof_2010,wang2009unifying}. Through the introduction of correction functions, which is equivalently viewed as a filtered DG correction field~\cite{zwanenburg_equivalence_2016,allaneau_connections_2011,Cicchino2020NewNorm}, ESFR recovers various high-order schemes alike the spectral difference~\cite{liu_spectral_2006} and spectral volume. 

There has been a concerted research effort to extend classical entropy stability arguments to high-order methods. The original work of Tadmor~\cite{tadmor1984skew} laid a foundation enabling high-order extensions, where Tadmor~\cite{tadmor1987numerical} proved that if the numerical flux satisfies the entropy condition from Harten~\cite{HartenSymmetric1983}, then the discretization is entropy stable---this was accomplished by introducing a weak condition on the numerical flux commonly referred to as the {\it Tadmor shuffle condition}. These notions were extended by LeFloch~\cite{lefloch2002fully,lefloch2000high} in the context of high-order finite difference stencils. In the last decade, these ideas were expanded to bounded domains by Fisher and co-authors~\cite{fisher2012high}, who combined the Summation-by-parts (SBP) framework with Tadmor's two-point flux functions to achieve entropy conservation. Extending the connection made by Fisher {\it et al}.~\cite{fisher2012high} between the SBP framework and Tadmor's shuffle condition, entropy conservative and stable methods have been successfully implemented in  DG~\cite{gassner2013skew,chan2018discretely,gassner2016split,chan2019skew}, FR~\cite{ranocha2016summation,abe2018stable} for only the DG case, unstructured methods~\cite{fernandez2014review,crean2018entropy,Crean2019Staggered}, extended high-order SBP forms~\cite{fisher2012high,fisher2013high,crean2018entropy,FriedrichEntropy2020,parsani2015entropyWall}, and residual distribution schemes~\cite{abgrall2018connection,abgrall2018general,abgrall2019reinterpretation,abgrall2021analysis}.

Recent developments by Cicchino~\textit{et al}.~\cite{CicchinoNonlinearlyStableFluxReconstruction2021,cicchino2022provably} proved that nonlinear stability for FR schemes can only be satisfied if the FR correction functions are applied to the nonlinear volume terms. This paper derives a general nonlinearly stable flux reconstruction (NSFR) framework for three-dimensional vector-valued problems in curvilinear coordinates. Specifically, we consider Euler's equations in this work. The first main contribution follows the works by Chan~\cite{chan2018discretely,chan2019discretely,chan2019skew} and Cicchino~\textit{et al}.~\cite{CicchinoNonlinearlyStableFluxReconstruction2021,cicchino2022provably} to derive NSFR that is free-stream preserving, globally conservative, and entropy conserving for general modal basis functions evaluated on quadrature rules with at least $2p-1$ strength, for any ESFR correction function. This is accomplished by formulating NSFR with the general modal skew-symmetric operators from Chan~\cite{chan2019skew}, then incorporating the ESFR correction functions on both the volume and surface hybrids terms through the modified mass matrix alike Cicchino and coauthors~\cite{CicchinoNonlinearlyStableFluxReconstruction2021,cicchino2022provably}. Apart from entropy conservation, kinetic energy conservation is also sought. Unfortunately, the inverse mapping from the kinetic energy variables to the conservative variables does not exist. This consequence leads to our proof that discrete kinetic energy stability cannot be preserved to machine precision for quadrature rules where the surface quadrature is not a subset of the volume quadrature.

We first introduce the notation used in this paper in Sec.~\ref{sec: Math notation}, where Sec.~\ref{sec: entropy notation} presents entropy conservation and Sec.~\ref{sec: comp notation} presents notation pertaining to computational quantities. Then, in Sec.~\ref{sec: DG} we derive the DG scheme, followed by ESFR in Sec.~\ref{sec: FR}, and lastly, we merge the concepts in Sec.~\ref{sec: NSFR}. In Sec.~\ref{sec: conserved prop} we prove that NSFR is free-stream preserving, globally conservative, and entropy conserving. In Sec.~\ref{sec: nonlinear stability} we prove that kinetic energy cannot be discretely conserved if the surface quadrature is not a subset of the volume quadrature. In Sec.~\ref{sec: TGV} we numerically verify all conserved quantities to machine precision for different FR correction functions on a non-symmetrically warped curvilinear grid. Lastly, in Sec.~\ref{sec: manfac sol}, we verify the orders of convergence for a manufactured solution on a three-dimensional, nonsymmetrically warped curvilinear grid.

As shown in Cicchino~\textit{et al}.~\cite{cicchino2022provably}, for curvilinear coordinates, the scheme requires that a dense matrix is inverted for every element. Motivated by the work by Chan and Wilcox~\cite{chan2019discretely,chan2017weight}, we present NSFR in a low-storage, weight-adjusted inverse form. This is achieved in Sec.~\ref{sec: weight adjusted} by introducing an ESFR projection operator, that projects onto the broken Sobolev-space, and deriving the weight-adjusted mass inverse through it. 
Then, we demonstrate that the operation can be further improved through sum-factorization techniques~\cite{orszag1979spectral} by introducing an FR projection operator. In Sec.~\ref{sec: manfac sol} we numerically verify that NSFR with a weight-adjusted mass inverse preserves the orders of convergence in curvilinear coordinates for a manufactured solution. 
Also, in Sec.~\ref{sec: TGV}, we numerically demonstrate that the weight-adjusted inverse preserves the stability properties for the weight-adjusted system. In Sec.~\ref{sec: comparison}, we compare a conservative nodal DG scheme, and overintegrated nodal DG scheme, and our proposed algorithm. We find that our proposed algorithm is computationally competitive with the nodal DG scheme, as compared to the overintegrated scheme. This result is dependent on the weight-adjusted mass matrix inversion, sum-factorization, and the novel sum-factorized Hadamard product evaluation from Cicchino and Nadarajah~\cite{CicchinoHadamardScaling2023}.

\section{Preliminaries}\label{sec: Math notation}

In this section, we present the notation that is used throughout the paper. First, we review concepts on entropy conservation. Then, we introduce notation pertaining to the computational setup.

\subsection{Systems of Equations}\label{sec: entropy notation}

Consider the scalar 3D conservation law,
\begin{equation}
\begin{aligned}
   & \frac{\partial}{\partial t}{u}_i\left(\bm{x}^c,t\right) +\nabla\cdot\bm{f}_i\left(\bm{u}\left(\bm{x}^c,t\right) \right)=0,\:\forall i\in[1,n_\text{state}], \: t\geq 0,\:\bm{x}^c\coloneqq[x \text{ }y\text{ }z]\in\Omega,\\
 &  {u}_i(\bm{x}^c,0)={u}_{i,0}\left(\bm{x}^c\right),
\end{aligned}\label{eq: gov eqn start}
\end{equation}

\noindent where $\bm{f}_i\left(\bm{u}\left(\bm{x}^c,t\right) \right)\in\mathbb{R}^{1\times d}$ stores the fluxes in each of the $d$ directions for the $i$-th equation of state, $n_\text{state}$ represents the number of state variables, and the superscript $c$ refers to Cartesian coordinates. In this paper row vector notation will be used.


\noindent {\it Smooth} solutions of Eq.~(\ref{eq: gov eqn start}) satisfy the entropy equality,

\begin{equation}
    \frac{\partial U}{\partial t} + \nabla \cdot \bm{F} = 0,\:\:\bm{F}=\bm{F}(u),\label{eq: entropy equality}
\end{equation}

\noindent where $U$ is a convex function of $u$, while {\it weak} solutions satisfy the entropy inequality,

\begin{equation}
    \frac{\partial U}{\partial t} + \nabla \cdot \bm{F} \leq 0.\label{eq: entropy inequality}
\end{equation}


In the context of entropy conservative numerical schemes, from Mock~\cite{mock1980systems} and Harten~\cite{HartenSymmetric1983}, it was shown that the symmetrization of the PDE by the entropy variables $\bm{v}=U^\prime(\bm{u}),\:\bm{v}\in\mathbb{R}^{1\times n_\text{state}}$, along with the convexity of the entropy function $U$, leads to the existence of an entropy flux function $\bm{F}(\bm{u})$ such that in each physical direction $k\in[1,d]$,

\begin{equation}
    \bm{v}\frac{\partial \bm{f}^k}{\partial \bm{u}}^T = \frac{d F^k(\bm{u})}{d \bm{u}},\quad \mbox{where,}\: \left(\frac{\partial \bm{f}^k}{\partial \bm{u}}^T\right)_{ij}=\frac{\partial f^k_i(\bm{u})}{\partial u_j},\:\forall i,j\in[1,n_\text{state}],\label{eq: entropy flux }  
\end{equation}

\noindent Integrating Eq.~(\ref{eq: entropy flux }) with respect to the conservative variables, and introducing the entropy potential $\bm{\psi}(\bm{v})$ such that $\frac{d\psi^k(\bm{v})}{d\bm{v}} = \bm{f}^k(\bm{u}(\bm{v}))$, we have~\cite[Eq. 1.13b]{HartenSymmetric1983},

\begin{equation}
    \psi^k(\bm{v}) = \bm{v}{\bm{f}^k}^T - F^k(\bm{u}),\:\forall k\in[1,d].\label{eq: entropy pot with flux relationship}
\end{equation}

\noindent Tadmor~\cite[Eq. 4.5a]{tadmor1987numerical} demonstrated that a numerical scheme is entropy conservative if,

\begin{equation}
    \bm{v} \Delta \left({\bm{f}^k}^T\right) = \Delta \left(F^k\right).\label{eq: entorpy cons flux cond}
\end{equation}

\noindent Unfortunately, Eq.~(\ref{eq: entorpy cons flux cond}) is a strong condition on the flux. To satisfy Eq.~(\ref{eq: entorpy cons flux cond}) in a weak sense, Tadmor~\cite{tadmor1987numerical} applied the $\Delta$ operator on Eq.~(\ref{eq: entropy pot with flux relationship}) to retrieve $\Delta \left(\psi^k\right) = \Delta \left(\bm{v}\right) {\bm{f}^k}^T  + \bm{v}\Delta \left({\bm{f}^k}^T\right) - \Delta \left(F^k\right)$, resulting in the equivalent condition~\cite[Eq. 4.5b]{tadmor1987numerical},

\begin{equation}
    \Delta \left(\bm{v}\right){\bm{f}^k}^T = \Delta \left(\psi^k\right).\label{eq: tadmor shuffle}
\end{equation}

Eq.~(\ref{eq: tadmor shuffle}) is known as the {\it Tadmor shuffle condition}. We introduce $\llbracket \bm{v} \rrbracket=\bm{v}_i - \bm{v}_j$ as the jump, and it has been shown that the two-point flux $\bm{f}^k_s(\bm{v}_i,\bm{v}_j)$,

\begin{equation}
\llbracket \bm{v} \rrbracket
   \bm{f}^k_s(\bm{v}_i,\bm{v}_j)^T = 
   \llbracket \psi^k \rrbracket
   ,\: \forall k\in[1,d]\label{eq: EC flux}
\end{equation}

\noindent is entropy conserving~\cite{tadmor1987numerical,chan2018discretely} in the sense that the numerical discretization conserves the equality\\ $\int_{\bm{\Omega}}\frac{\partial U}{\partial t}d\bm{\Omega} + \int_{\bm{\Gamma}}\left(\bm{v} \bm{f}^T - \bm{\psi}\right)\cdot \hat{\bm{n}} \:d\bm{\Gamma} = 0$ exactly.

In this paper, we consider the three-dimensional unsteady Euler equations,

\begin{equation}
\begin{split}
 &   \frac{\partial \bm{W}^T}{\partial t} + \nabla \cdot \bm{f}\left(\bm{W}\right)^T = \bm{0}^T,\\
 & \bm{W} = \begin{bmatrix}
 \rho,& \rho u,& \rho v ,& \rho w ,& \rho e
 \end{bmatrix},\\
& \bm{f}_1 = \begin{bmatrix}
     \rho u,& \rho u^2 + p,& \rho u v, & \rho u w ,& \left(\rho e + p\right)u
 \end{bmatrix},\\
 &\bm{f}_2 = \begin{bmatrix}
     \rho v,& \rho uv,& \rho  v^2  + p, & \rho v w ,& \left(\rho e + p\right)v
 \end{bmatrix},\\
 &\bm{f}_3 = \begin{bmatrix}
     \rho w,& \rho uw,& \rho vw , & \rho  w^2 + p ,& \left(\rho e + p\right)w
 \end{bmatrix},
\end{split}
\end{equation}

\noindent where $\rho e = \frac{p}{\gamma -1} + \frac{1}{2}\rho\left(u^2+v^2+w^2\right)$, and $\rho,\: u,\: v,\:w,\:p\:e\:\gamma $ are the density, velocity, pressure, specific total energy, and adiabatic coefficient respectively.

\subsection{Computational Space}\label{sec: comp notation}

To discretely solve Eq.~(\ref{eq: gov eqn start}), we need to introduce notations with respect to the computational space and basis functions. Since we discretely represent each equation of state separately, we will remove the boldface on $u$ to indicate that it is for a single state variable.

The computational domain $\bm{\Omega}^h$ is partitioned into $M$ non-overlapping elements, $\bm{\Omega}_m$, where the domain is represented by the union of the elements, \textit{i.e.}
\begin{equation}
    \bm{\Omega} \simeq \bm{\Omega}^h \coloneqq \bigcup_{m=1}^M \bm{\Omega}_m.
    \nonumber
\end{equation}
Each element $m$ has a surface denoted by $\bm{\Gamma}_m$. The global approximation, $u^h(\bm{x}^c,t)$, is constructed from the direct sum of each {\color{black}{local approximation, $u_m^h(\bm{x}^c,t)$}}, \textit{i.e.} 
\begin{equation}
    u(\bm{x}^c,t) \simeq u^h(\bm{x}^c,t)=\bigoplus_{m=1}^M u_m^h(\bm{x}^c,t).
    \nonumber
\end{equation}
Throughout this paper, all quantities with a subscript $m$ are specifically unique to the element $m$.
{\color{black}On each element, we represent the solution with $N_p$ linearly independent modal or nodal basis functions of a maximum order of $p$; where, $N_p\coloneqq(p+1)^d$. 
The solution representation is},
$
     u_m^h(\bm{x}^c,t) \coloneqq  \sum_{i=1}^{N_p} {\chi}_{m,i}(\bm{x}^c)\hat{u}_{m,i}(t)
    \nonumber
$, where $\hat{u}_{m,i}(t)$ are the modal coefficients for the solution. 
The elementwise residual {for the governing equation~(\ref{eq: gov eqn start})} is,
\begin{equation}
     R_m^h(\bm{x}^c,t)=\frac{\partial}{\partial t}u_m^h(\bm{x}^c,t) + \nabla \cdot \bm{f}(u_m^h(\bm{x}^c,t)).\label{eq: residual}
\end{equation}

\noindent The basis functions in each element are defined as,
\begin{equation}
     \bm{\chi}(\bm{x}^c) \coloneqq [\chi_{1}(\bm{x}^c)\text{, }\chi_{2}(\bm{x}^c)\text{, } \dots\text{, } \chi_{N_p}(\bm{x}^c)] = \bm{\chi}({x})\otimes \bm{\chi}({y}) \otimes \bm{\chi}({z}) \in \mathbb{R}^{1\times N_p},
\end{equation}

\noindent where $\otimes$ is the tensor product. Importantly, in curvilinear elements, the basis functions are not polynomial in physical space, but the basis functions are polynomial in reference space. This concept was explored in great detail in Cicchino~{\it et al}.~\cite[Sec. 3]{cicchino2022provably}, as well as proofs for rate of convergence in Botti~\cite{botti2012influence} and Moxey~\textit{et al}.~\cite{moxey2019interpolation}.

The physical coordinates are mapped to the reference element $ \bm{\xi}^r \coloneqq \{ [\xi \text{, } \eta \text{, }\zeta]:-1\leq \xi,\eta,\zeta\leq1 \}$ by

\begin{equation}
      \bm{x}_m^c( \bm{\xi}^r)\coloneqq  \bm{\Theta}_m( \bm{\xi}^r)=
     \sum_{i=1}^{N_{t,m}} {\Theta}_{m,i}( \bm{\xi}^r)\hat{ \bm{x}}_{m,i}^c,
\end{equation}
where $ {\Theta}_{m,i}$ 
are the mapping shape functions of the $N_{t,m}$ physical mapping control points $\hat{\bm{x}}_{m ,i}^c $.

To transform Eq.~(\ref{eq: residual}) to the reference basis, {\color{black}as in refs}~\cite{gassner2018br1,manzanero2019entropy,kopriva2006metric,thomas1979geometric,vinokur2002extension}, we introduce the physical 

\begin{equation}
    \bm{a}_j \coloneqq \frac{\partial \bm{x}^c}{\partial \xi^j},\text{ } j=1,2,3
    \nonumber
\end{equation}

\noindent and reference 

\begin{equation}
    \bm{a}^j \coloneqq \nabla\xi^j,\text{ } j=1,2,3
    \nonumber
\end{equation}

\noindent vector bases. We then introduce the determinant of the metric Jacobian as

\begin{equation}
  J^\Omega \coloneqq  |\bm{J}^\Omega| =  \bm{a}_1\cdot (\bm{a}_2\times \bm{a}_3),
\end{equation}

\noindent and the metric Jacobian cofactor matrix as ~\cite{gassner2018br1,zwanenburg_equivalence_2016,manzanero2019entropy,kopriva2019free},

\begin{equation}
    \bm{C}^T \coloneqq J^\Omega(\bm{J}^\Omega)^{-1}=\begin{bmatrix}J^\Omega\bm{a}^1\\
    J^\Omega\bm{a}^2\\
    J^\Omega\bm{a}^3 \end{bmatrix}
    = \begin{bmatrix}
    J^\Omega\bm{a}^\xi\\
    J^\Omega\bm{a}^\eta\\
    J^\Omega\bm{a}^\zeta
    \end{bmatrix}.
\end{equation}

The metric cofactor matrix is formulated by the \enquote{conservative curl} form from~\cite[Eq. 36]{kopriva2006metric} {\color{black}so }as to discretely satisfy the Geometric Conservation Law (GCL) 

\begin{equation}
   \sum_{i=1}^{3}\frac{\partial (J^\Omega (\bm{a}^i)_n)}{\partial \xi^i}=0\text{, }n=1,2,3 \Leftrightarrow \sum_{i=1}^{3}\frac{\partial}{\partial \xi^i}
    (\bm{C})_{ni}=0\text{, }n=1,2,3
    \Leftrightarrow \nabla^r\cdot(\bm{C})=\bm{0},
    \label{eq: GCL}
\end{equation}

\noindent {\color{black}for a fixed mesh}, where $(\; )_{ni}$ represents the $n^{\text{th}}$ row, $i^{\text{th}}$ column component of a matrix. The exact implementation of the metric cofactor matrix is extensively detailed in Cicchino~{\it et al}.~\cite[Sec. 5]{cicchino2022provably} 

{\color{black}Having established the transformations mapping the physical to the reference coordinates on each element, the differential volume and surface elements can be defined as,}

\begin{equation}
    d\bm{\Omega}_m = J_m^\Omega d\bm{\Omega}_r\text{, similarly }d\bm{\Gamma}_m = J_m^\Gamma d\bm{\Gamma}_r.\label{eq: diff element def}
\end{equation}

\noindent The reference flux for each element $m$ is defined as

\begin{equation}
    \bm{f}_m^r = \bm{C}_m^T\cdot \bm{f}_m 
   \Leftrightarrow \bm{f}_{m,j}^r =\sum_{i=1}^d (\bm{C}_{m}^T)_{ji}\bm{f}_{m,i} \Leftrightarrow \bm{f}_m^r
   =\bm{f}_{m}\bm{C}_m,
   \label{eq: contravariant flux def}
\end{equation}

\noindent where the dot product notation for tensor-vector operations is introduced. The relationship between the physical and reference unit normals is given as~\cite[Appendix B.2]{zwanenburg_equivalence_2016},

\begin{equation}
    \hat{\bm{n}}_m = \frac{1}{J_m^\Gamma}\bm{C}_m\cdot \hat{\bm{n}}^r 
    =\frac{1}{J_m^\Gamma} \hat{\bm{n}}^r \bm{C}_m^T,\label{eq: normals def}
\end{equation}

\noindent for a water-tight mesh. 
 Additionally, the definition of the divergence operator derived from divergence theorem in curvilinear coordinates can be expressed as~\cite[Eq. (2.22) and (2.26)]{gassner2018br1},

\begin{equation}
    \nabla\cdot \bm{f}_m 
    =\frac{1}{J_m^\Omega} \nabla^r \cdot\Big( \bm{f}_m\bm{C}_m \Big)= \frac{1}{J_m^\Omega} \nabla^r \cdot \bm{f}_m^r,\label{eq: divergence def}
\end{equation}

\noindent and the gradient of a scalar as~\cite[Eq. (2.21)]{gassner2018br1},

\begin{equation}
    \nabla {\chi} = \frac{1}{J_m^\Omega} \bm{C}_m \cdot \nabla^r \chi = \frac{1}{J_m^\Omega} \Big(\nabla^r \chi \Big) \bm{C}_m^T.\label{eq: gradient def}
\end{equation}


Thus, substituting Eq.~(\ref{eq: divergence def}) into Eq.~(\ref{eq: residual}), the reference elementwise residual can be expressed as,

\begin{equation}
    R_m^{h,r}(\bm{\xi}^r,t)\coloneqq R_m^{h}(\bm{\Theta}_m(\bm{\xi}^r),t)= \frac{\partial}{\partial t}u_m^h(\bm{\Theta}_m(\bm{\xi}^r),t) + \frac{1}{J_m^\Omega}\nabla^r \cdot \bm{f}^r(u_m^h(\bm{\Theta}_m(\bm{\xi}^r),t)).\label{eq: residual reference}
\end{equation}

Lastly, since the basis functions are polynomial in the reference space, we introduce $\bm{\chi}\left(\bm{\xi}^r\right)$, and they discretely satisfy discrete integration-by-parts for quadrature rules exact for at least $2p-1$ polynomials,

\begin{equation}
    \int_{\bm{\Omega}_r} \nabla^r{\chi}_i\left(\bm{\xi}^r\right) 
    {\chi}_j\left(\bm{\xi}^r\right) d\bm{\Omega}_r
    + \int_{\bm{\Omega}_r} {\chi}_i\left(\bm{\xi}^r\right) 
    \nabla^r{\chi}_j\left(\bm{\xi}^r\right) d\bm{\Omega}_r
    = \int_{\bm{\Gamma}_r} {\chi}_i\left(\bm{\xi}^r\right){\chi}_j\left(\bm{\xi}^r\right) \hat{\bm{n}}^r d\bm{\Gamma}_r,\:\forall i,j\in[1,N_p].\label{eq: SBP}
\end{equation}

\noindent Eq.~(\ref{eq: SBP}) is commonly referred to as the summation-by-parts (SBP) property.
\section{Nonlinearly Stable Flux Reconstruction}

In this section, we will present our nonlinearly stable FR scheme. To arrive at it, we first present DG, then ESFR. Following the motivation from~\cite{CicchinoNonlinearlyStableFluxReconstruction2021,cicchino2022provably}, we then incorporate the entropy stable framework within the stiffness operator to arrive at NSFR.

\subsection{Discontinuous Galerkin} \label{sec: DG}

The discontinuous Galerkin approach is obtained by multiplying the reference residual Eq.~(\ref{eq: residual reference}) by a test function which is chosen to be the basis function and integrating in physical space. Then, we apply integration-by-parts in the reference space on the divergence of the flux. Since the solution and flux are both discontinuous across the face, we introduce a numerical surface flux $\bm{f}^*_m$ to treat the resulting Riemann problem. The continuous weak DG is,

\begin{equation}
\begin{split}
    \int_{\bm{\Omega}_r} {\chi}_i(\bm{\xi}^r) J_m^\Omega\bm{\chi}(\bm{\xi}^r) \frac{d}{d t} \hat{\bm{u}}_m(t)^T d \bm{\Omega}_r
    -\int_{\bm{\Omega}_r} \nabla^r{\chi}_i(\bm{\xi}^r) 
    \cdot {{\bm{f}}_m^r} d\bm{\Omega}_r 
    +\int_{\bm{\Gamma}_r} {\chi}_i(\bm{\xi}^r) 
    \hat{\bm{n}}^r\bm{C}_m^T \cdot \bm{f}_m^{*}(u_m^h(\bm{\Theta}_m(\bm{\xi}^r),t))d \bm{\Gamma}_r = {0},\\ \:\forall i=1,\dots,N_p.
\end{split}\label{eq: variational Weak DG}
\end{equation}

The nonlinear reference flux $\bm{f}_m^r$ is computed through Eq.~(\ref{eq: contravariant flux def}), where the physical flux $\bm{f}_m = f(\bm{u}(\bm{\xi}_v^r,t)$ is evaluated directly from the solution at the integration points $\bm{\xi}_v^r$. Discretely, to evaluate the integrals in Eq.~(\ref{eq: variational Weak DG}), we utilize quadrature rules, where $\bm{\xi}_v^r$ represents the volume quadrature nodes, and $\bm{\xi}_{f,k}^r$ represents the surface quadrature node on the face $f\in[1,N_f]$ with facet cubature node $k\in[1,N_{fp}]$. To arrive at the equivalent strong form, we first need to pay special attention to the nonlinear flux. Since the reference flux is evaluated directly on the volume cubature nodes,  we need to project it onto a polynomial basis with the order of the number of quadrature nodes. For example, if $\bm{u}^h$ is order $p$, and it is integrated on $p+4$ quadrature nodes, the flux basis--$\bm{\phi}(\bm{\xi}^r)$ needs to be of at least order $p+3$ to avoid aliasing errors when representing the flux. The easiest choice of basis functions for the flux basis is to use Lagrange polynomials collocated on the quadrature nodes since a collocated Lagrange basis has an identity projection operator. By collocated, we mean that the Lagrange polynomial is evaluated on the same nodes, the quadrature nodes, as the polynomial it is constructed from, making $l_{ij}=\delta_{ij}$, where $i$ is the polynomial basis number and $j$ the nodal number.

Thus, applying discrete integration by parts on Eq.~(\ref{eq: variational Weak DG}) results in,

\begin{equation}
    \bm{M}_m \frac{d}{d t} \hat{\bm{u}}_m(t)^T
    + \bm{\chi}(\bm{\xi}_v^r)^T \bm{W} \nabla^r\bm{\phi}(\bm{\xi}_v^r) \cdot \hat{\bm{f}}_m^r(t)^T
    + \sum_{f=1}^{N_f}\sum_{k=1}^{N_{fp}}
    \bm{\chi}(\bm{\xi}_{fk}^r)^T W_{fk} \hat{\bm{n}}^r \cdot
    \left[\bm{f}_m^{*,r} -\bm{\phi}(\bm{\xi}_{fk}^r)\hat{\bm{f}}_m^r(t)^T  \right]=\bm{0}^T,\label{eq: cons DG strong}
\end{equation}

\noindent where $\bm{W}$ is a diagonal matrix storing the quadrature weights and the discrete mass matrix is,

\begin{equation}
    \begin{aligned}
      (\bm{M}_m)_{ij}\approx\int_{\Omega_r}J_m^\Omega {\chi}_i(\bm{\xi}^r){\chi}_j(\bm{\xi}^r)d\Omega_r \to \bm{M}_m= \bm{\chi}(\bm{\xi}_v^r)^T\bm{W}\bm{J}_m\bm{\chi}(\bm{\xi}_v^r),
    \end{aligned}
\end{equation}

\noindent with $\bm{J}_m$ as a diagonal matrix storing the determinant of the metric Jacobian at quadrature nodes.

\subsection{Energy Stable Flux Reconstruction}\label{sec: FR}

As initially proposed by H.T. Huynh~\cite{huynh_flux_2007}, in an ESFR framework, the reference flux is composed of a discontinuous and a corrected component,
\begin{equation}
    \bm{f}^r(u_m^h(\bm{\Theta}_m(\bm{\xi}^r),t))\coloneqq 
    \bm{f}^{D,r}(u_m^h(\bm{\Theta}_m(\bm{\xi}^r),t)) + \sum_{f=1}^{N_f}\sum_{k=1}^{N_{f_{p}}} \bm{g}^{f,k}(\bm{\xi}^r)[\hat{\bm{n}}^r\cdot ({\bm{f}}_m^{*,r}-{\bm{f}}_m^r)].\label{eq: ESFR corr flux}
\end{equation}

\noindent For three-dimensions, the vector correction functions $\bm{g}^{f,k}(\bm{\xi}^r)\in\mathbb{R}^{1\times d}$ associated with face $f$, facet cubature node $k$ in the reference element, are defined as the tensor product of the $p+1$ order one-dimensional correction functions ($\bm{\chi}_{p+1}$ stores a basis of order $p+1$), with the corresponding $p$-th order basis functions in the other reference directions.
\begin{equation}
\begin{split}
    \bm{g}^{f,k}(\bm{\xi}^r) = [\Big(\bm{\chi}_{p+1}(\xi)\otimes \bm{\chi}({\eta}) \otimes \bm{\chi}({\zeta})\Big)\Big(\hat{\bm{g}}^{f,k}_1\Big)^T , \Big(\bm{\chi}({\xi})\otimes \bm{\chi}_{p+1}(\eta)\otimes \bm{\chi}({\zeta})\Big)\Big(\hat{\bm{g}}^{f,k}_2\Big)^T,\Big(\bm{\chi}({\xi}) \otimes \bm{\chi}({\eta}) \otimes \bm{\chi}_{p+1}(\zeta)\Big) \Big(\hat{\bm{g}}^{f,k}_3\Big)^T ] \\
    = [{g}^{f,k}_1(\bm{\xi}^r), {g}^{f,k}_2(\bm{\xi}^r), {g}^{f,k}_3(\bm{\xi}^r)],
    \end{split}
\end{equation}

\noindent such that 
\begin{equation}
    \bm{g}^{f,k}(\bm{\xi}_{f_i, k_j}^r) \cdot \hat{\bm{n}}^r_{f_i,k_j} = \begin{cases}
    1, \;\;\text{ if } f_i = f,\text{ and } k_j=k\\
    0, \;\;\text{ otherwise.}
    \end{cases}\label{eq: ESFR conditions}
\end{equation}

Coupled with the symmetry condition $g^L(\xi^r)=-g^R(-\xi^r)$ to satisfy Eq.~(\ref{eq: ESFR conditions}), the one-dimensional ESFR fundamental assumption from~\cite{vincent_new_2011} is, 

\begin{equation}
    \int_{-1}^{1} \nabla^r {\chi}_i({\xi}^r)
      {g}^{f,k}({\xi}^r) d\xi  
    - c\frac{\partial^{p} {\chi}_i({\xi}^r) ^T}{\partial \xi^p}
   \frac{\partial^{p+1} {g}^{f,k}({\xi}^r)}{\partial \xi^{p+1}} = {0},\: \forall i=1,\dots,N_p,\label{eq: 1D ESFR fund assumpt}
\end{equation}

\noindent and similarly for the other reference directions.

%
%

Akin to~\cite{castonguay2012newTRI,williams_energy_2014,cicchino2022provably}, consider introducing the differential operator,

\begin{equation}
\begin{split}
&\text{2D:} \indent \partial^{(s,v)}=\frac{\partial^{s+v}}{\partial \xi^s\partial \eta^v},\: \text{such that } s=\{ 0,p\},\: v=\{ 0,p\},\: s+v\geq p ,\\
&\text{3D:}\indent     \partial^{(s,v,w)}=\frac{\partial^{s+v+w}}{\partial \xi^s\partial \eta^v\partial \zeta^w},\: \text{such that } s=\{ 0,p\},\: v=\{ 0,p\},\: w=\{ 0,p\},\: s+v+w\geq p,\label{eq: partial deriv ESFR def}
\end{split}
\end{equation}

\noindent with its corresponding correction parameter

\begin{equation}
\begin{split}
&\text{2D:} \indent c_{(s,v)}=c_{1D}^{(\frac{s}{p}+\frac{v}{p})},\\
&\text{3D:}\indent     c_{(s,v,w)}=c_{1D}^{(\frac{s}{p}+\frac{v}{p}+\frac{w}{p})}.
    \end{split}
\end{equation}

\noindent {\color{black} Note that the total degree is $d\times p$ for a tensor product basis that is of order $p$ in each direction.}

\noindent For example,

\begin{equation}
    \partial^{(0,p,0)} = \frac{\partial^p}{\partial \eta^p},\: 
     c_{(0,p,0)} = c_{1D},\:
    \partial^{(p,0,p)}= \frac{\partial^{2p}}{\partial \xi^p \partial \zeta^p},\: c_{(p,0,p)} = c_{1D}^2,\:
    \partial^{(p,p,p)}= \frac{\partial^{3p}}{\partial \xi^p \partial \eta^p \partial \zeta^p},\: c_{(p,p,p)} = c_{1D}^3.
    \nonumber
\end{equation}

\noindent Since $\int_{\bm{\Omega}_r}\partial^{(s,v,w)}\bm{\chi}(\bm{\xi}^r)^T\partial^{(s,v,w)}\Big(\nabla^r\bm{\chi}(\bm{\xi}^r)\Big) d\bm{\Omega}_r$ composes of the complete broken Sobolev-norm for each $s$, $v$, $w$~\cite{sheshadri2016stability,cicchino2022provably}, the tensor product ESFR fundamental assumption, that recovers the VCJH~\cite{vincent_insights_2011} schemes exactly for linear elements is defined as,

\begin{equation}
        \int_{\bm{\Omega}_r} \nabla^r {\chi}_i(\bm{\xi}^r) \cdot
     \bm{g}^{f,k}(\bm{\xi}^r) d\bm{\Omega}_r  
    - \sum_{s,v,w}c_{(s,v,w)}\partial^{(s,v,w)} {\chi}_i(\bm{\xi}^r) 
   \partial^{(s,v,w)} \Big( \nabla^r \cdot \bm{g}^{f,k}(\bm{\xi}^r)\Big) = {0},\: \forall i=1,\dots, N_p,\label{eq: ESFR fund assumpt}
\end{equation}

\noindent where $\sum_{s,v,w}$ sums over all possible $s$, $v$, $w$ combinations in Eq.~(\ref{eq: partial deriv ESFR def}).

To discretely represent the divergence of the correction functions, we introduce the correction field \newline ${h}^{f,k}(\bm{\xi}^r)\in P_{3p}(\bm{\Omega}_r)$ associated with the face $f$ cubature node $k$ as,

\begin{equation}
    {h}^{f,k}(\bm{\xi}^r) = \bm{\chi}(\bm{\xi}^r)\Big(\hat{\bm{h}}^{f,k}\Big)^T
    =\nabla^r \cdot \bm{g}^{f,k}(\bm{\xi}^r).
\end{equation}
To arrive at the ESFR strong form, we substitute the ESFR reference flux, Eq.~(\ref{eq: ESFR corr flux}), into the elementwise reference residual, Eq.~(\ref{eq: residual reference}), project it onto the polynomial basis, and evaluate at cubature nodes,

\begin{equation}
    \bm{\chi}(\bm{\xi}_v^r)\frac{d}{dt}\hat{\bm{u}}_m(t)^T 
    +\bm{J}_m^{-1}\nabla^r\bm{\chi}(\bm{\xi}_v^r) \cdot \hat{\bm{f}}_{m}^{D,r}(t)^T
    +\bm{J}_m^{-1}\sum_{f=1}^{N_f}\sum_{k=1}^{N_{fp}} \bm{\chi}(\bm{\xi}_v^r)\Big(\hat{\bm{h}}^{f,k}\Big)^T [\hat{\bm{n}}^r\cdot ({\bm{f}}_m^{*,r}-{\bm{f}}_m^r)]   =\bm{0}^T.
    \label{eq:ESFR strong}
\end{equation}

%
%

For implementation in a pre-existing DG framework, Allaneau and Jameson~\cite{allaneau_connections_2011} showed that ESFR can be expressed as a filtered DG scheme in one-dimension. Zwanenburg and Nadarajah~\cite{zwanenburg_equivalence_2016} proved that ESFR can be expressed as a filtered DG scheme for general three-dimensional curvilinear coordinates on mixed element types. Expanding the reference ESFR filter in Zwanenburg and Nadarajah~\cite{zwanenburg_equivalence_2016} allows ESFR schemes to be seen as a DG-type scheme with a modified norm~\cite{Cicchino2020NewNorm}. Only by viewing ESFR as DG with a modified norm can a nonlinearly stable form be achieved~\cite{CicchinoNonlinearlyStableFluxReconstruction2021,cicchino2022provably} since the determinant of the metric Jacobian gets embedded in the norm, and the entropy stability application is embedded in the stiffness operator. The equivalence to view ESFR through a modified norm is dependent on~\cite[Lemma 2]{cicchino2022provably}, where it was proven that the ESFR correction operator has no influence on the conservative volume term when the inverse of the mass matrix is applied. To present ESFR as a DG scheme with a modified mass matrix, we introduce the ESFR correction operator as,

\begin{equation}
    \begin{split}
    &(\bm{K}_m)_{ij} \approx \sum_{s,v,w } c_{(s,v,w)} \int_{ {\Omega}_r} J_m^\Omega \partial^{(s,v,w)} \chi_i(\bm{\xi}^r) 
    \partial^{(s,v,w)}\chi_j(\bm{\xi}^r) d {\Omega_r}\\
    \to& 
    \bm{K}_m = 
     \sum_{s,v,w} c_{(s,v,w)}
    \partial^{(s,v,w)}\bm{\chi}(\bm{\xi}_v^r) ^T \bm{W}\bm{J}_m
   \partial^{(s,v,w)}\bm{\chi}(\bm{\xi}_v^r)
   =\sum_{s,v,w } c_{(s,v,w)}\Big(\bm{D}_\xi^s \bm{D}_\eta^v\bm{D}_\zeta^w \Big)^T\bm{M}_m\Big(\bm{D}_\xi^s \bm{D}_\eta^v\bm{D}_\zeta^w \Big),
    \end{split}\label{eq:Km}
\end{equation}

\noindent where $\bm{D}_\xi^s=\Big(\bm{M}^{-1}\bm{S}_\xi\Big)^s$ is the strong form differential operator raised to the power $s$, and similarly for the other reference directions. 

Therefore, recasting Eq.~(\ref{eq: cons DG strong}) as an ESFR scheme through the modified-norm framework~\cite{CicchinoNonlinearlyStableFluxReconstruction2021,cicchino2022provably} results in,

\begin{equation}
      \left(\bm{M}_m + \bm{K}_m\right) \frac{d}{d t} \hat{\bm{u}}_m(t)^T
    + \bm{\chi}(\bm{\xi}_v^r)^T \bm{W} \nabla^r\bm{\phi}(\bm{\xi}_v^r) \cdot \hat{\bm{f}}_m^r(t)^T
    + \sum_{f=1}^{N_f}\sum_{k=1}^{N_{fp}}
    \bm{\chi}(\bm{\xi}_{fk}^r)^T W_{fk} \hat{\bm{n}}^r \cdot
    \left[\bm{f}_m^{*,r} -\bm{\phi}(\bm{\xi}_{fk}^r)\hat{\bm{f}}_m^r(t)^T  \right]=\bm{0}^T.\label{eq: ESFR strong}
\end{equation}

\noindent Comparing Eq.~(\ref{eq: ESFR strong}) with Eq.~(\ref{eq: cons DG strong}), we note that all of the FR contributions arise from the modified mass matrix. 
\subsection{Nonlinearly Stable Flux Reconstruction}\label{sec: NSFR}

As in Chan~\cite{chan2018discretely,chan2019discretely,chan2019skew}, we utilize the general differential operator to recast Eq.~(\ref{eq: ESFR strong}) in a skew-symmetric two-point flux differencing form~\cite[Eq. (15)]{chan2019skew},

\begin{equation}
    \left(\bm{M}_m +\bm{K}_m\right)\frac{d}{d t} \hat{\bm{u}}_m(t)^T
     + \left[ \bm{\chi}(\bm{\xi}_v^r)^T\:  \bm{\chi}(\bm{\xi}_f^r)^T\right]
     \left[ \left(\Tilde{\bm{Q}}-\Tilde{\bm{Q}}^T\right) \odot \bm{F}_{m}^r \right]\bm{1}^T
     + \sum_{f=1}^{N_f}\sum_{k=1}^{N_{fp}}
    \bm{\chi}(\bm{\xi}_{fk}^r)^T W_{fk} \hat{\bm{n}}^r \cdot
    \bm{f}_m^{*,r}=\bm{0}^T,\label{eq: NSFR}
\end{equation}

\noindent where,

\begin{equation}
    \Tilde{\bm{Q}} - \Tilde{\bm{Q}}^T = 
    \begin{bmatrix}
        \bm{W}\nabla^r\bm{\phi}(\bm{\xi}_v^r) - \nabla^r\bm{\phi}(\bm{\xi}_v^r) ^T \bm{W}
        &  \sum_{f=1}^{N_f} \bm{\phi}(\bm{\xi}_f^r)^T \bm{W}_f \text{diag}(\hat{\bm{n}}^r_f) \\
    -  \sum_{f=1}^{N_f}  \bm{W}_f \text{diag}(\hat{\bm{n}}^r_f) \bm{\phi}(\bm{\xi}_f^r) & \bm{0} 
    \end{bmatrix}\in \mathbb{R}^{\left(N_v + N_{fp}\right)\times\left(N_v + N_{fp}\right)\times{d}}\label{eq: skew-symm qtilde oper}
\end{equation}

\noindent is the general hybridized skew-symmetric stiffness operator involving both volume and surface quadrature evaluations~\cite{chan2019skew} that has each $\Tilde{\bm{Q}}$ satisfying the SBP-like property

\begin{equation}
    \Tilde{\bm{Q}} + \Tilde{\bm{Q}}^T = 
    \begin{bmatrix}
        \bm{0} & \bm{0} \\
        \bm{0} & \bm{W}_f \text{diag}(\hat{\bm{n}}_f^r)
    \end{bmatrix}.\label{eq: SBP Qtilde}
\end{equation}

 \noindent The hybrid skew-symmetric stiffness operator is constructed solely from the flux basis since the flux basis discretely satisfies the SBP property on the quadrature nodes. Here, $\bm{\phi}(\bm{\xi}_f^r)$ be the matrix storing the flux basis evaluated at all surface quadrature nodes on the surface $f$. We let $\odot$ represent a Hadamard product with a 
 component-wise dot product on the vectors of length dimension. This is accomplished by performing a Hadamard product of size $\left(N_v + N_{fp}\right)\times\left(N_v +  N_{fp}\right)$ in each of the $d$ directions, then performing the summation on each $d$ direction from the dot product. Also, $\bm{F}_m^r$ is the matrix storing the reference two-point flux values,

\begin{equation}
     {  \left( \bm{F}^r_m\right)_{ij}} = 
     {\bm{f}_s\left(\Tilde{\bm{u}}_m(\bm{\xi}_{i}^r),\Tilde{\bm{u}}_m(\bm{\xi}_{j}^r)\right)}
     {\left(
       \frac{1}{2}\left(\bm{C}_m(\bm{\xi}_{i}^r)+ \bm{C}_m(\bm{\xi}_{j}^r) \right)\right)},\:\forall\: 1\leq i,j\leq N_v + N_{fp}.
\end{equation}

We chose to incorporate the splitting of the metric cofactor matrix~\cite{cicchino2022provably} within forming the reference two-point flux rather than incorporating it within a \enquote{physical} skew-symmetric stiffness operator~\cite{chan2019discretely,chan2019skew}. Although both forms are mathematically equivalent, the former allows us to exploit the tensor product structure of the reference basis stiffness operator to perform the Hadamard product at order $\mathcal{O}\left(n^{d+1} \right)$~\cite{CicchinoHadamardScaling2023} whereas the latter does not. We provide a detailed report on evaluations of Hadamard products at $\mathcal{O}\left(n^{d+1} \right)$ using tensor product basis in the technical report~\cite{CicchinoHadamardScaling2023}.

 The choice of two-point flux ${\bm{f}_s\left(\Tilde{\bm{u}}_m(\bm{\xi}_{i}^r),\Tilde{\bm{u}}_m(\bm{\xi}_{j}^r)\right)}$ dictates the physically relevant conserving properties of the scheme. 
 For an entropy conserving scheme, we choose a two-point flux that satisfies the Tadmor shuffle condition Eq.~(\ref{eq: tadmor shuffle}),

\begin{equation}
      \left(\bm{v}\left(\bm{\xi}_{i}^r\right) - \bm{v}\left(\bm{\xi}_{j}^r\right)\right) {\bm{f}_s\left(\Tilde{\bm{u}}_m(\bm{\xi}_{i}^r),\Tilde{\bm{u}}_m(\bm{\xi}_{j}^r)\right)^T}
       = \bm{\psi}\left(\bm{v}\left(\bm{\xi}_{i}^r\right)\right) 
       - \bm{\psi}\left(\bm{v}\left(\bm{\xi}_{j}^r\right)\right)
       ,\:\forall\: 1\leq i,j\leq N_v + N_{fp},\label{eq: Tadmor shuffle quad node}
\end{equation}

\noindent with the entropy-projected conservative variables $\Tilde{\bm{u}}_m(\bm{\xi}^r)$~\cite{chan2018discretely}. The entropy projected conservative variables are computed by interpolating the conservative solution to the volume quadrature nodes, and evaluating the entropy variables on the quadrature nodes. Then the entropy variables are projected onto the solution basis to obtain the $p$-th order modal coefficients. Lastly, the modal coefficients are then interpolated to the volume and surface quadrature nodes where the entropy-projected conservative variables are obtained by doing the inverse of the mapping. This process is critical because, in Eq.~(\ref{eq: Tadmor shuffle quad node}), the entropy potential $\bm{\psi}$ is a function of the entropy variables, not of the conservative variables. Thus, we need to discretely satisfy it with the entropy variables at the quadrature nodes. Also, for the stability condition, the residual is left multiplied by the modal coefficients of the entropy variables---thus we need the entropy variables projected on the solution basis. The entropy projected variables are summarized by,

\begin{equation}
    \Tilde{\bm{u}}\left( \bm{\xi}^r\right)
    = \bm{u}\left(\bm{\chi}(\bm{\xi}^r)\hat{\bm{v}}^T \right),\:
    \hat{\bm{v}}^T = \bm{\Pi}\left(\bm{v}\left(\bm{\chi}\left(\bm{\xi}_v^r\right)\hat{\bm{u}}^T\right) \right).
\end{equation}

The solution to the Tadmor shuffle condition in Eq.~(\ref{eq: Tadmor shuffle quad node}) is $\bm{f}_\text{EC} = \frac{\Delta \psi}{\Delta v}$. For Burgers' equation, there is a unique flux that satisfies the Tadmor shuffle equation~\cite{tadmor1984skew}, but for Euler, there are different possibilities depending on the variables chosen for expressing the jumps~\cite[Sec. 4.5]{chandrashekar2013kinetic}. Ismail and Roe~\cite{ismail2009affordable} were the first to introduce a change of variables to arrive at an explicit form for an entropy conserving flux. Motivated to conserve both entropy and kinetic energy, Chandrashekar~\cite{chandrashekar2013kinetic} derived an entropy conserving and kinetic energy preserving flux by considering the variables $\rho, \:u,\:$and $\beta=\frac{\rho}{2p}=\frac{1}{2RT}$, where $R$ is the universal gas constant and $T$ is the temperature. 
Importantly noted by both Ismail and Roe~\cite{ismail2009affordable} and Chandrashekar~\cite{chandrashekar2013kinetic}, for an entropy dissipative flux, the jump in the entropy variables must be recovered. If the jump in the conservative variables is employed, then for the Euler equations, the scheme is not purely entropy dissipative~\cite[Sec. 5]{chandrashekar2013kinetic}. Lastly, Ranocha~\cite{ranocha2018generalised} presented a systematic framework to constructing an entropy conserving and kinetic energy preserving flux from different combinations of thermodynamic variables for the Euler equations.

\section{Weight-Adjusted Inverse of Flux Reconstruction Mass Matrix}\label{sec: weight adjusted}

By immediate inspection, inverting $\bm{M}_m+\bm{K}_m$ on-the-fly is costly since they are both fully dense matrices in curvilinear coordinates. Instead, our goal is to make use of the tensor product structure and use sum-factorization~\cite{orszag1979spectral} techniques for efficient, low storage evaluations on-the-fly. 

Similar to Chan and Wilcox~\cite[Eq. (27)]{chan2019discretely}, we let $\bm{u}_J^T=\bm{J}_m\bm{u}_m^T$. 
We need to solve the matrix system,

\begin{equation}
    (\bm{M}_{1/J} +\bm{K}_{1/J}) \bm{u}_J^T = (\bm{M} +\bm{K})\bm{u}^T,
\end{equation}

\noindent where 

\begin{equation}
    \begin{split}
        \bm{M}_{1/J}=\bm{\chi}(\bm{\xi}_v^r)^T\bm{W}\bm{J}_m^{-1}\bm{\chi}(\bm{\xi}_v^r),\\
        \bm{K}_{1/J}=\sum_{s,v,w } c_{(s,v,w)}\left(\bm{D}_\xi^s \bm{D}_\eta^v\bm{D}_\zeta^w \right)^T\bm{M}_{1/J}\left(\bm{D}_\xi^s \bm{D}_\eta^v\bm{D}_\zeta^w \right).
    \end{split}\label{eq: over J mass}
\end{equation}

Thus, we can express the respective weight-adjusted system FR mass matrix and inverse of the FR mass matrix as,

\begin{equation}
\begin{split}
    \left(\bm{M}_m + \bm{K}_m\right) \approx \left(\bm{M} + \bm{K}\right) \left( \bm{M}_{1/J} +\bm{K}_{1/J}\right)^{-1}\left(\bm{M} + \bm{K}\right), \\
    \left(\bm{M}_m + \bm{K}_m\right)^{-1} \approx \left(\bm{M} + \bm{K}\right)^{-1} \left( \bm{M}_{1/J} +\bm{K}_{1/J}\right)\left(\bm{M} + \bm{K}\right)^{-1}.
    \end{split}\label{eq: weight adj inv}
\end{equation}

It is important to note that $\left(\bm{M} + \bm{K}\right)^{-1}$ has the same local value for each element. Therefore, computing $\left(\bm{M}_m + \bm{K}_m\right)^{-1}$ via Eq.~(\ref{eq: weight adj inv}) only depends on inverting the diagonal matrix storing the determinant of the Jacobian on-the-fly.

As detailed by Chan and Wilcox~\cite{chan2019discretely}, the accuracy of the weight-adjusted inverse approximation is solely dependent on the accuracy of the projection. In Cicchino~\textit{et al}.~\cite{cicchino2022provably}, the ESFR projection operator was demonstrated to maintain the orders of convergence up to the upper limit found by Castonguay~\cite[Fig. 3.6]{castonguay_phd}. Thus, we can make further improvements by introducing the ESFR projection operator in reference space,

\begin{equation}
\begin{split}
      &  \Tilde{\bm{\Pi}}(\bm{\xi}_v^r) =  \Tilde{\bm{\Pi}}({\xi}_v)
        \otimes  \Tilde{\bm{\Pi}}({\eta}_v)
        \otimes  \Tilde{\bm{\Pi}}({\zeta}_v),\\
     &    \Tilde{\bm{\Pi}}({\xi}_v) = \left(\bm{M}(\xi_v)+\bm{K}(\xi_v)\right)^{-1} \bm{\chi}\left({\xi}_v\right)^T\bm{W}(\xi_v),
\end{split}\label{eq: Fr projection}
\end{equation}

\noindent where $\bm{M}(\xi_v)+\bm{K}(\xi_v)$ is the one-dimensional modified mass matrix, $ \bm{\chi}\left({\xi}_v\right)$ is the one-dimensional basis function, and $\bm{W}(\xi_v)$ stores the 1D quadrature weights. 

Therefore, using the reference ESFR projection operator in Eq.~(\ref{eq: Fr projection}), the weight-adjusted inverse to the ESFR modified mass matrix is,

\begin{equation}
    \left(\bm{M}_m + \bm{K}_m\right)^{-1}
    \approx
    \Tilde{\bm{\Pi}}(\bm{\xi}_v^r) \left(\bm{W}\bm{J}_m\right)^{-1}\Tilde{\bm{\Pi}}(\bm{\xi}_v^r)^T.\label{eq: weight-adj ESFR mass}
\end{equation}

Using Eq.~(\ref{eq: weight-adj ESFR mass}) to approximate the inverse reduces the computational cost to inverting a diagonal matrix, $\bm{W}\bm{J}_m$, on-the-fly, while being able to exploit the tensor product structure of the reference ESFR projection operator. Lastly, we need to show that the ESFR weight-adjusted inverse preserves the orders of convergence.

\begin{thm}
    The weight-adjusted inverse for the ESFR mass matrix in Eq.~(\ref{eq: weight-adj ESFR mass}) preserves the order of convergence $\mathcal{O}(h^{p+1})$.
\end{thm}

\begin{proof}
From Chan and Wilcox~\cite[Thm 4]{chan2019discretely}, the weight-adjusted inverse is of order $h^{\text{min}(r,p+1)+1}$ for \\$u\in W^{r,2}\left(\bm{\Omega}_m\right)$. With the choice of ESFR correction functions from Vincent~\textit{et al}.~\cite{vincent_new_2011}, and by rewriting the correction functions through a filtered DG form~\cite{allaneau_connections_2011,zwanenburg_equivalence_2016}, the ESFR projection operator is of the same order as the DG projection operator for values of $c\leq c_{+}$. This was numerically shown by Castonguay~\cite{castonguay_phd} for the ESFR correction functions and by Cicchino~\textit{et al}.~\cite{cicchino2022provably} for the ESFR projection of the volume split-form divergence in curvilinear coordinates. Therefore, since the ESFR projection operator is of the same order of accuracy as the DG projection operator, the ESFR weight-adjusted inverse holds the same order of accuracy as the DG result in Chan and Wilcox~\cite{chan2019discretely}.
\end{proof}

In Sec.~\ref{sec: manfac sol}, we demonstrate numerically that the numerical scheme with the weight-adjusted inverse for ESFR discretizations maintains the correct orders of accuracy 
for values $c\leq c_{+}$~\cite{vincent_new_2011,castonguay_phd,cicchino2022provably}. 
\section{Scalable Evaluations}\label{sec: scalable eval}

With the intent of the high-order solver being used on next generation hardware, on both CPUs and GPUs, we wish that the solver scales at the lowest order possible, and has a low memory footprint. In this section, we briefly review sum-factorization techniques to drastically reduce the memory footprint and flop count for evaluating the NSFR discretization in Eq.~(\ref{eq: NSFR}). For this section, we will let $n=p+1$ and $d$ still represent the dimension of the operators. By observation, the flops in the NSFR discretization Eq.~(\ref{eq: NSFR}) are dominated by dense matrix-vector multiplications, along with a dense Hadamard product. We resolve the scaling issue for the matrix-vector multiplications with sum-factorization~\cite{orszag1979spectral}, and the Hadamard product with the sum-factorized Hadamard product from Cicchino and Nadarajah~\cite[Thm. 2.1]{CicchinoHadamardScaling2023}. Following the roofline model~\cite{kolev2021efficient}, all operations would ideally have a large arithmetic intensity (A.I.), $\text{A.I.}=\frac{\text{flops}}{\text{bytes}}$, with it being a function of $n$.

Let's consider interpolating the modal coefficients $\hat{\bm{u}}$ to the volume quadrature nodes by the basis functions $\bm{\chi}\left(\bm{\xi}_v^r\right)$. Since, $\bm{\chi}\left(\bm{\xi}_v^r\right) = \bm{\chi}\left({\xi}_v^r\right)\otimes\bm{\chi}\left({\eta}_v^r\right)\otimes \bm{\chi}\left({\zeta}_v^r \right) \in \bm{M}_{n^{d}\times n^{d}}\left(\mathbb{R}\right)$ is the tensor product of the $n$-sized one-dimensional basis functions operators $\bm{\chi}\left({\xi}_v^r\right)$, then $\bm{u}_v^T = \bm{\chi}\left(\bm{\xi}_v^r\right) \hat{\bm{u}}^T$ is evaluated directly in $n^{2d}$ flops. Orszag~\cite{orszag1979spectral} observed that the tensor product structure can be exploited, and the matrix-vector multiplication can be evaluated in each direction independently. The flop count for sum-factorization is $dn^{d+1}$. As extensively detailed in Karniadakis and Sherwin~\cite[Chapters 3,4]{karniadakis2013spectral}, by using quadrilateral and hexahedral reference elements, the basis operations can straightforwardly use sum-factorization. For triangular, tetrahedral, prismatic, and pyramidic-based elements, Karniadakis and Sherwin~\cite[Chapters 3,4]{karniadakis2013spectral} derived orthogonal tensor product basis functions to exploit sum-factorization in high-order codes. 

We will present the sum-factorization operations as a sequence of matrix-matrix multiplications alike Cantwell~\textit{et al}.~\cite{cantwell2011h} to leverage optimizations in the BLAS subsystem. We let the $\xi$-direction run the fastest, then $\eta$, and $\zeta$ running the slowest. First, we rearrange the modal coefficient's of $\hat{\bm{u}}\to \hat{\bm{u}}^{n_\xi :n_\eta n_\zeta}\in \bm{M}_{n_\xi \times n_\eta n_\zeta}\left(\mathbb{R}\right)$, where $n_\xi$ is the one-dimensional basis size in the $\xi$-direction, similarly for $n_\eta$ and $n_\zeta$. The transformation makes $\hat{\bm{u}}^{n_\xi :n_\eta n_\zeta}$ a matrix where its rows vary in the $\xi$-direction, and its columns vary in $\eta$ and $\zeta$. This allows us to perform a one-dimensional matrix-vector multiplication in $\xi$, at $n^2$ flops, and perform it $n^{d-1}$ times for the $\eta$ and $\zeta$ combinations in the columns. Alike Cantwell~\textit{et al}.~\cite[Sec. 2.3.1]{cantwell2011h}, following the same steps for the other directions arrives at,

\begin{equation}
    \begin{split}
        \bm{q}_0^{n_\xi : n_\eta n_\zeta} = \bm{\chi}\left(\xi_v^r\right)\hat{\bm{u}}^{n_\xi :n_\eta n_\zeta},\\
        \bm{q}_1^{n_\eta : n_\zeta n_\xi } = \bm{\chi}\left(\eta_v^r\right){\bm{q}_0}^{n_\eta :n_\zeta n_\xi},\\
        \bm{q}_2^{n_\zeta : n_\xi n_\eta } = \bm{\chi}\left(\zeta_v^r\right){\bm{q}_1}^{n_\zeta :n_\xi n_\eta},\\
        \bm{u}_v = \bm{q}_2^{n_\xi n_\eta n_\zeta : 1}.
    \end{split}\label{eq: sum-factorization}
\end{equation}

Each line in Eq.~(\ref{eq: sum-factorization}) has $n^{2}n^{d-1}=n^{d+1}$ flops, and then it is repeated for the $d$-lines, arriving at $dn^{d+1}$ flops. The memory footprint involves loading the $d$ one-dimensional basis operator matrices of size $n^2$, loading the $n^d$ vector, writing the final $n^{d}$ vector, and writing/loading the $n^{d}$ sub-vectors $d$-times. Thus, the arithmetic intensity for Eq.~(\ref{eq: sum-factorization}) is $\text{A.I.}=\frac{dn^{d+1}}{(d+2)n^{d}+dn^2}$. A detailed analysis of the strong scaling and roofline model for high-order methods, such as continuous Galerkin, DG, and hybrid DG, while exploiting sum-factorization on quadrilateral/hexahedral meshes is found in Kronbichler and Wall~\cite{kronbichler2018performance} and Moxey~\textit{et al}.~\cite{moxey2020efficient} for triangular based elements. 

For the evaluation of the mass inverse, and other multi-step operators, each step is evaluated consecutively. For example, let's consider applying the weight-adjusted mass matrix inverse approximation, Eq.~(\ref{eq: weight-adj ESFR mass}), on the right-hand-side. We would first use sum-factorization on multiplying $\Tilde{\bm{\Pi}}\left(\bm{\xi}_v^r\right)^T $ to the right-hand-side at $dn^{d+1}$ flops. Then we multiply the inverse of a diagonal operator storing the determinant of the Jacobian multiplied with the quadrature weights at $n^d$ flops. Lastly, we use sum-factorization to multiply $ \Tilde{\bm{\Pi}}\left(\bm{\xi}_v^r\right)$ for an additional $dn^{d+1}$ flops. Rather than building a mass matrix at $n^{2d}$ flops, inverting it at $n^{3d}$ flops, and performing one matrix-vector operation at $n^{2d}$ flops, we used the {\it matrix-free} sum-factorization approach to evaluate it in three separate steps, at a total of $2dn^{d+1}+n^d$ flops.

For the Hadamard product in Eq.~(\ref{eq: NSFR}), we use the sum-factorized Hadamard product from Cicchino and Nadarajah~\cite[Thm. 2.1]{CicchinoHadamardScaling2023}. The volume-volume Hadamard product from the upper left block of Eq.~(\ref{eq: skew-symm qtilde oper}) is evaluated in $dn^{d+1}$ flops. The surface-volume Hadamard products from the upper right and lower left blocks in Eq.~(\ref{eq: skew-symm qtilde oper}) are each evaluated in $dn^d$ flops. 
Since sum-factorization evaluates a matrix-vector product in $dn^{d+1}$ flops, the computational cost difference between the divergence of the flux in conservative strong DG and the Hadamard product in NSFR is in the evaluation of the two-point flux. 

In Sec.~\ref{sec: comparison}, for the Taylor-Green vortex problem on a nonsymmetrically warped curvilinear mesh, we numerically verify the scaling at order $\mathcal{O}\left(n^{d+1}\right)$ for the NSFR discretization in Eq.~(\ref{eq: NSFR}) as compared to the conservative DG scheme in Eq.~(\ref{eq: cons DG strong}) with and without overintegration. We also numerically compare the wall clock time for the flow simulation. 

\section{Conserved Properties of NSFR}\label{sec: conserved prop}

In this section we provide the three main theorems of this work: free-stream preservation, global conservation, and entropy stability for NSFR in a weight-adjusted framework.

\subsection{Free-stream Preservation}\label{sec: freestream pres}

The first conserved quantity that NSFR preserves is free-stream preservation. If the free-stream is not preserved, then the nonlinear metric terms would introduce cross-wind into the flow that drastically destroy the fidelity of the solver~\cite{visbal2002use,kopriva2006metric}. Thus, for curvilinear coordinates, it is necessary that the discretization is provably free-stream preserving.

\begin{thm}
    The NSFR discretization in Eq.~(\ref{eq: NSFR}) with the weight-adjusted low-storage mass matrix inverse from Eq.~(\ref{eq: weight-adj ESFR mass}) is free-stream preserving.
\end{thm}

\begin{proof}
    Similar to Cicchino~\textit{et al}.~\cite[Sec. 5.1]{cicchino2022provably}, we substitute $\bm{f}_m=\bm{\alpha}=\text{constant}$ and $\frac{d\hat{\bm{u}}_m(t)}{dt}=\bm{0}$ into Eq.~(\ref{eq: NSFR}),

    \begin{equation}
       \left[ \bm{\chi}(\bm{\xi}_v^r)^T\:  \bm{\chi}(\bm{\xi}_f^r)^T\right]
     \left[ \left(\Tilde{\bm{Q}}-\Tilde{\bm{Q}}^T\right) \odot \bm{F}_{m}^r \right]\bm{1}^T
     + \sum_{f=1}^{N_f}\sum_{k=1}^{N_{fp}}
    \bm{\chi}(\bm{\xi}_{fk}^r)^T W_{fk} \hat{\bm{n}}^r \cdot
    \bm{\alpha}\bm{C}_m(\bm{\xi}_{fk}^r) .\label{eq: freestream 1}
    \end{equation}


\noindent Substituting the SBP property Eq.~(\ref{eq: SBP Qtilde}) for $\Tilde{\bm{Q}}^T$ into Eq.~(\ref{eq: freestream 1}) results in,

\begin{equation}
     \left[ \bm{\chi}(\bm{\xi}_v^r)^T\:  \bm{\chi}(\bm{\xi}_f^r)^T\right]
     \left[ 2 \Tilde{\bm{Q}} \odot \bm{F}_{m}^r \right]\bm{1}^T
     + \sum_{f=1}^{N_f}\sum_{k=1}^{N_{fp}}
    \bm{\chi}(\bm{\xi}_{fk}^r)^T W_{fk} \hat{\bm{n}}^r \cdot
   \left[ \bm{\alpha}\bm{C}_m(\bm{\xi}_{fk}^r) - \bm{\alpha}\bm{C}_m(\bm{\xi}_{fk}^r)\right],
\end{equation}

\noindent where the surface integral cancels off. For the Hadamard product, we note that the physical flux is constant, rendering the reference flux $\left(\bm{F}_m^r\right)_{ij} = \frac{1}{2}\bm{\alpha}\left(\bm{C}_m(\bm{\xi}_{i}^r) +\bm{C}_m(\bm{\xi}_{j}^r) \right)$ equal to the central flux with respect to the metric cofactor matrix. Using Fisher and Carpenter~\cite[Thm 3.1 and 3.2]{fisher2012high} and Gassner~\textit{et al}.~\cite[Eq. (3.5)]{gassner2016split}, we can expand the Hadamard product as,

\begin{equation}
    \begin{split}
&       \left[ \bm{\chi}(\bm{\xi}_v^r)^T\:  \bm{\chi}(\bm{\xi}_f^r)^T\right]
     \left[ 2 \Tilde{\bm{Q}} \odot \bm{F}_{m}^r \right]\bm{1}^T
     = \bm{\chi}(\bm{\xi}_v^r)^T \bm{W}\nabla^r\bm{\phi}\left(\bm{\xi}_v^r\right) \cdot\left(\bm{\alpha} \bm{C}_m\left(\bm{\xi}_v^r\right)\right)\\
&     - \bm{\chi}(\bm{\xi}_v^r)^T  \sum_{f=1}^{N_f}\sum_{k=1}^{N_{fp}}
    \left[\bm{\phi}(\bm{\xi}_{fk}^r)^T W_{fk} \hat{\bm{n}}^r \cdot
   \bm{\phi}(\bm{\xi}_{fk}^r)\left(\bm{\alpha}\bm{C}_m(\bm{\xi}_{v}^r)\right)
   -\frac{1}{2}\bm{\phi}(\bm{\xi}_{fk}^r)^T W_{fk} \hat{\bm{n}}^r \cdot \left(\bm{\alpha}\bm{C}_m(\bm{\xi}_{f,k}^r)\right)\right]\\
&   +\frac{1}{2}\sum_{f=1}^{N_f}\sum_{k=1}^{N_{fp}} \bm{\chi}(\bm{\xi}_{f,k}^r)^TW_{fk} \hat{\bm{n}}^r \cdot
   \bm{\phi}(\bm{\xi}_{fk}^r)\left(\bm{\alpha}\bm{C}_m(\bm{\xi}_{v}^r)\right).
    \end{split}
\end{equation}

\noindent The volume term on the right-hand side vanishes if the GCL in Eq.~(\ref{eq: GCL}) is satisfied discretely. In Cicchino~\textit{et al}.~\cite[Sec. 5]{cicchino2022provably}, we provided a detailed review to construct the metric terms to discretely satisfy GCL with consistent surface normals. The review in Cicchino~\textit{et al}.~\cite[Sec. 5]{cicchino2022provably} first summarized Kopriva's~\cite{kopriva2006metric} derivation of the conservative and invariant curl forms, then followed the work by Abe~\textit{et al}.~\cite{abe2015freestream} to ensure consistency on the surfaces for different flux and grid nodes. By using the consistency condition from the mapping shape functions from the grid nodes to the flux nodes, then
$ \bm{\phi}(\bm{\xi}_{fk}^r)\left(\bm{\alpha}\bm{C}_m(\bm{\xi}_{v}^r)\right) = \bm{\alpha}\bm{C}_m(\bm{\xi}_{f,k}^r)$, and $\bm{\phi}(\bm{\xi}_{fk}^r)\bm{\chi}(\bm{\xi}_{v}^r)=\bm{\chi}(\bm{\xi}_{fk}^r)$, which sees the surface terms eliminated. Therefore, the NSFR discretization is free-stream preserving for vector-valued conservation laws.

\end{proof}

%
%

\subsection{Global Conservation}\label{sec: conservation}

A crucial property for discretizations that approximate conservation laws is that they discretely satisfy conservation. Continuously, the conservation property is obtained by performing divergence theorem on the flux, $\int_{\bm{\Omega}} \frac{\partial u}{\partial t} d\bm{\Omega} =- \int_{\bm{\Omega}} \nabla \cdot \bm{f} d\bm{\Omega} = -\int_{\bm{\Gamma}} \bm{f} \cdot \hat{\bm{n}} d\bm{\Gamma} $.

\begin{thm}
    The NSFR discretization in Eq.~(\ref{eq: NSFR}) with the weight-adjusted low-storage mass matrix inverse from Eq.~(\ref{eq: weight-adj ESFR mass}) is globally conserving.
\end{thm}

\begin{proof}
    For each equation of state we left multiply Eq.~(\ref{eq: NSFR}) by $\hat{\bm{1}}$, such that $\bm{\chi}(\bm{\xi}_v^r)\hat{\bm{1}}^T = \bm{1}^T $,

    \begin{equation}
        \hat{\bm{1}} \left(\bm{M}_m +\bm{K}_m\right)\frac{d}{d t} \hat{\bm{u}}_m(t)^T
     + \bm{1}
     \left[ \left(\Tilde{\bm{Q}}-\Tilde{\bm{Q}}^T\right) \odot \bm{F}_{m}^r \right]\bm{1}^T
     + \sum_{f=1}^{N_f}\sum_{k=1}^{N_{fp}}
     W_{fk} \hat{\bm{n}}^r \cdot
    \bm{f}_m^{*,r}.
    \end{equation}

    Alike Chan~\cite[Thm. 3]{chan2019skew}, $\left(\Tilde{\bm{Q}}-\Tilde{\bm{Q}}^T\right) \odot \bm{F}_{m}^r$ is skew-symmetric, so the volume terms vanish. Therefore, 

    \begin{equation}
        \hat{\bm{1}} \left(\bm{M}_m +\bm{K}_m\right)\frac{d}{d t} \hat{\bm{u}}_m(t)^T = -\sum_{f=1}^{N_f}\sum_{k=1}^{N_{fp}}
     W_{fk} \hat{\bm{n}}^r \cdot
    \bm{f}_m^{*,r},
    \end{equation}
    \noindent and the scheme is globally conservative.
\end{proof}
\subsection{Nonlinear Stability}\label{sec: nonlinear stability}

{\color{black}The last theorem of this section pertains to nonlinear stability. When a scheme is provably discretely nonlinearly stable, then, based from the work of Lyapunov, the approximate solution in bounded within a norm to the true solution. 
Hence, the discrete solution {\color{black}remains stable} which implies guaranteed robustness. It is important to note that stability does not imply convergence.}

\begin{thm}
    The NSFR discretization in Eq.~(\ref{eq: NSFR}) with the weight-adjusted low-storage mass matrix inverse from Eq.~(\ref{eq: weight-adj ESFR mass}) is discretely entropy conserving if the two-point flux satisfies the Tadmor shuffle condition. Thus, the NSFR discretization in Eq.~(\ref{eq: NSFR}) is nonlinearly stable.
\end{thm}

\begin{proof}
    We left multiply Eq.~(\ref{eq: NSFR}) by the modal coefficients of the entropy variables evaluated at the quadrature nodes, and sum over all of the states,

    \begin{equation}
       \sum_{n=1}^{n_\text{state}}\left[ \hat{\bm{v}}_n \left(\bm{M}_m+\bm{K}_m \right)\frac{d}{d t} \hat{\bm{u}}_{m,n}(t)^T
        +\hat{\bm{v}}_n\left[ \bm{\chi}(\bm{\xi}_v^r)^T\:  \bm{\chi}(\bm{\xi}_f^r)^T\right]
     \left[ \left(\Tilde{\bm{Q}}-\Tilde{\bm{Q}}^T\right) \odot \bm{F}_{m,n}^r \right]\bm{1}^T
     +\hat{\bm{v}}_n \sum_{f=1}^{N_f}\sum_{k=1}^{N_{fp}}
    \bm{\chi}(\bm{\xi}_{fk}^r)^T W_{fk} \hat{\bm{n}}^r \cdot
    \bm{f}_{m,n}^{*,r}\right].\label{eq: entropy var discr}
    \end{equation}

\noindent It is important to note that Eq.~(\ref{eq: entropy var discr}) is not the same as only applying the $\left(\bm{M}_m+\bm{K}_m\right)^{-1}$ on the surface as commonly used in the FR literature~\cite{vincent_new_2011,castonguay_energy_2013,huynh_reconstruction_2009,allaneau_connections_2011,zwanenburg_equivalence_2016,ranocha2016summation,ranocha2017extended,abe2018stable}. To demonstrate this, consider comparing the stability condition for NSFR,

\begin{equation}
    \begin{split}
     &   \text{NSFR Stability: }\\
      &  \sum_{n=1}^{n_\text{state}}\Big[ \hat{\bm{v}}_n \left(\bm{M}_m+\bm{K}_m \right)\frac{d}{d t} \hat{\bm{u}}_{m,n}(t)^T
        +
       \hat{\bm{v}}_n \left(\bm{M}_m+\bm{K}_m \right)
        \left(\bm{M}_m+\bm{K}_m \right)^{-1}\left[ \bm{\chi}(\bm{\xi}_v^r)^T\:  \bm{\chi}(\bm{\xi}_f^r)^T\right]
     \left[ \left(\Tilde{\bm{Q}}-\Tilde{\bm{Q}}^T\right) \odot \bm{F}_{m,n}^r \right]\bm{1}^T\\
 &    \hspace{0.7cm}+\hat{\bm{v}}_n\left(\bm{M}_m+\bm{K}_m \right)
        \left(\bm{M}_m+\bm{K}_m \right)^{-1} \sum_{f=1}^{N_f}\sum_{k=1}^{N_{fp}}
    \bm{\chi}(\bm{\xi}_{fk}^r)^T W_{fk} \hat{\bm{n}}^r \cdot
    \bm{f}_{m,n}^{*,r}\Big],
    \end{split}\label{eq: NSFR stability}
    \end{equation}

\noindent versus a classical ESFR scheme in a flux differencing form,

\begin{equation}
    \begin{split}
  &  \text{ESFR Flux Differencing Stability: }\\
       & \sum_{n=1}^{n_\text{state}}\Big[ \hat{\bm{v}}_n \left(\bm{M}_m+\bm{K}_m \right)\frac{d}{d t} \hat{\bm{u}}_{m,n}(t)^T
        +
        \hat{\bm{v}}_n\left(\bm{M}_m+\bm{K}_m \right)
        \left(\bm{M}_m \right)^{-1}\left[ \bm{\chi}(\bm{\xi}_v^r)^T\:  \bm{\chi}(\bm{\xi}_f^r)^T\right]
     \left[ \left(\Tilde{\bm{Q}}-\Tilde{\bm{Q}}^T\right) \odot \bm{F}_{m,n}^r \right]\bm{1}^T\\
   &  \hspace{0.7cm}+\hat{\bm{v}}_n\left(\bm{M}_m+\bm{K}_m \right)
        \left(\bm{M}_m+\bm{K}_m \right)^{-1} \sum_{f=1}^{N_f}\sum_{k=1}^{N_{fp}}
    \bm{\chi}(\bm{\xi}_{fk}^r)^T W_{fk} \hat{\bm{n}}^r \cdot
    \bm{f}_{m,n}^{*,r}\Big].
    \end{split}\label{eq: ESFR stability}
\end{equation}

{\color{black}By considering FR as DG with a modified mass matrix $\bm{M}_m+\bm{K}_m$, we directly incorporate the influence of the FR correction functions through $\bm{K}_m$ on the non-conservative volume terms. This allows Eq.~(\ref{eq: entropy var discr}) to implicitly have $\left(\bm{M}_m+\bm{K}_m\right)\left(\bm{M}_m+\bm{K}_m\right)^{-1}$ cancel on the volume-surface hybrid term as seen in Eq.~(\ref{eq: NSFR stability}). That was not observed in Ranocha~\textit{et al}.~\cite{ranocha2016summation,ranocha2017extended} nor Abe~\textit{et al}.~\cite{abe2018stable} since the ESFR flux differencing volume term does not have the $\bm{K}_m$ operator vanish, which renders it unstable, as seen by $\bm{K}_m\bm{M}_m^{-1}$ in Eq.~(\ref{eq: ESFR stability}).}

The next step highlights the importance of using the entropy-projected variables. As in, Chan~\cite{chan2018discretely}, we will denote $\Tilde{\bm{v}}=\bm{\chi}\left(\bm{\xi}^r\right)\hat{\bm{v}}^T$, then the volume-surface hybrid term becomes,

\begin{equation}
  \sum_{n=1}^{n_\text{state}}  \hat{\bm{v}}_n\left[ \bm{\chi}(\bm{\xi}_v^r)^T\:  \bm{\chi}(\bm{\xi}_f^r)^T\right]
     \left[ \left(\Tilde{\bm{Q}}-\Tilde{\bm{Q}}^T\right) \odot \bm{F}_{m,n}^r \right]\bm{1}^T
     =  \sum_{n=1}^{n_\text{state}} \sum_{i,j=1}^{N_{v}+N_f N_{fp}}
     \left(\Tilde{\bm{Q}}\right)_{ij} \left(\Tilde{{v}}_{n,i}-\Tilde{{v}}_{n,j}\right) \bm{f}_{s,n}^r\left(\Tilde{\bm{u}}_m(\bm{\xi}_{i}^r),\Tilde{\bm{u}}_m(\bm{\xi}_{j}^r)\right).\label{eq: surf vol entropy hybrid}
\end{equation}

{\color{black}Chan~\cite{chan2018discretely} arrived at Eq.~(\ref{eq: surf vol entropy hybrid}) by expanding the volume hybrid term in a flux differencing form that sums over all quadrature nodes~\cite[Eq. (71)]{chan2018discretely}. Then the Tadmor shuffle condition was substituted for the change of entropy variables and two-point flux to recover the entropy potential on the surface through integration-by-parts~\cite[Eq. (72) and (73)]{chan2018discretely}. The substitution for the Tadmor shuffle condition by Chan~\cite[Eq. (72) and (73)]{chan2018discretely} was only possible because the two-point flux was constructed by the entropy projected variables. Since the Tadmor shuffle condition is a function of the entropy variables in Eq.~(\ref{eq: Tadmor shuffle quad node}), the two-point flux has to satisfy $\bm{f}_{s,n} = \frac{\bm{\psi}_{n,i}-\bm{\psi}_{n,j}}{\Tilde{v}_{n,i} - \Tilde{v}_{n,j}}$, and thus the conservative variables used to construct $\bm{f}_s$ need to be mapped from $\Tilde{\bm{v}}$. Making the substitution renders, 
}



    \begin{equation}
          \sum_{n=1}^{n_\text{state}} \hat{\bm{v}}_n \left(\bm{M}_m+\bm{K}_m \right)\frac{d}{d t} \hat{\bm{u}}_{m,n}(t)^T
         = \sum_{n=1}^{n_\text{state}} \sum_{f=1}^{N_f}\sum_{k=1}^{N_{fp}}\left( 
         \bm{\psi}_n\left(\bm{\xi}_{fk}^r\right) - {v}_n\left(\bm{\xi}_{fk}^r\right)\bm{f}_{m,n}^{*,r}
         \right)\cdot \hat{\bm{n}}^r .
    \end{equation}

\noindent
  Using appropriate boundary conditions, and the choice of $\bm{f}_{m,n}^{*,r}\left(\Tilde{\bm{u}}_\text{left}(\bm{\xi}_{fk}^r),\Tilde{\bm{u}}_\text{right}(\bm{\xi}_{fk}^r)\right)$ as an entropy conserving flux based on entropy projected variables from the left and right of the face, completes the proof for discrete entropy conservation within the $(\bm{M}_m+\bm{K}_m)$-norm.
\end{proof}

To incorporate entropy dissipation, as extensively detailed by Chandrashekar~\cite[Sec. 5 and 6]{chandrashekar2013kinetic}, careful treatment needs to be considered. By directly considering the jump in the conservative variables with a positive scalar does not provably guarantee entropy dissipation because the jump in energy is not entropy dissipative~\cite[Sec. 5.2]{chandrashekar2013kinetic}. Thus, we make use of the Roe dissipation~\cite{roe1981approximate} which can provably be recast with the jump of the entropy variables to provably add entropy dissipation~\cite{ismail2009affordable,chandrashekar2013kinetic}. 

An important consequence of the need for projected entropy variables to satisfy the Tadmor shuffle condition in the volume is that a scheme cannot be discretely kinetic energy preserving if the surface quadrature is not a subset of the volume quadrature. For example, if the scheme is integrated on Gauss-Legendre quadrature nodes it cannot be exactly discretely kinetic energy preserving, whereas if it is integrated on Gauss-Legendre-Lobatto volume quadrature it can be kinetic energy preserving.

\begin{lemma}
    A high-order discretization cannot be exactly discretely kinetic energy preserving if the surface quadrature nodes are not a subset of the volume quadrature nodes.\label{lem: kin energy}
\end{lemma}

\begin{proof}

    We begin by considering the kinetic energy as $K.E. =\frac{1}{2}\rho\left(u^2+v^2+w^2\right)$. Similarly to entropy, the kinetic energy variables are,

    \begin{equation}
       \bm{v}_{K.E.} =  \frac{\partial K.E.}{\partial \bm{u}} = 
       \left[ -\frac{u^2+v^2+w^2}{2},\:u,\:v,\:w,\:0\right],\label{eq: kin energy var}
    \end{equation}
\noindent and the kinetic energy potential is $\bm{\psi}^k = 0,\:\forall k\in [1,d]$. Since the fifth value of $\bm{v}_{K.E.}$ is $0$, an inverse mapping from $\bm{v}_{K.E.}\to\bm{u}$ does not exist. Considering we have a two-point flux that satisfies kinetic energy preservation, we need the surface numerical flux evaluated with the entropy projected variables on the face $f$ facet cubature node $k$,
    $\Tilde{\bm{u}}\left(\bm{\xi}_{fk}^r\right) =\bm{u}\left(\bm{\chi}\left(\bm{\xi}_{fk}^r\right)\hat{\bm{v}}^T_{K.E.}\right)$. Although we can evaluate the interpolation of the kinetic energy variable to the face, we cannot perform the inverse mapping to extract the conservative variables. If the surface quadrature nodes are a subset of the volume quadrature nodes, then the interpolation and projection operators become an identity, 
    and the inverse mapping is never needed to evaluate $\Tilde{\bm{u}}\left(\bm{\xi}_{fk}^r\right)$, allowing for discrete kinetic energy conservation only for the case when the surface quadrature nodes are a subset of the volume quadrature nodes (for example Gauss-Lobatto-Legendre).

\end{proof}

We numerically verify Lemma~\ref{lem: kin energy} in Sec.~\ref{sec: TGV}, where kinetic energy is conserved to machine precision when integrated on Gauss-Lobatto-Legendre quadrature nodes, and it is dissipated when integrated on Gauss-Legendre quadrature nodes.
\section{Results}\label{sec: results}

In this section, we use the open-source Parallel High-order Library for PDEs (\texttt{PHiLiP})~\cite{shi2021full}. We first perform a grid study to verify the orders of convergence with the weight-adjusted framework. Then, we numerically verify the nonlinear stability properties from Section~\ref{sec: conserved prop}. 

For all tests, we use a 4-stage, fourth-order Runge-Kutta timestepping scheme with an adaptive timestep based on the maximum wavespeed in the domain with a CFL$=0.1$. 
We let the adiabatic constant $\gamma=1.4$ for all tests. For the NSFR-EC schemes, \enquote{EC} refers to entropy conserving, we use Chandrashekar's~\cite{chandrashekar2013kinetic} entropy conserving flux with Ranocha's~\cite{ranocha2022preventing} pressure fix for kinetic energy preservation 
for the two-point flux; while \enquote{GL} refers to integrating on Gauss-Legendre nodes and \enquote{LGL} refers to integrating on Lobatto-Gauss-Legendre quadrature. All schemes have the conservative solution modal coefficients on LGL nodes. We use the weight-adjusted mass matrix inverse for all schemes except the conservative DG. For the curvilinear case, the metric terms are constructed by $p+1$ order polynomials detailed in Cicchino~\textit{et al}.~\cite{cicchino2022provably}. Lastly, when the scheme is \enquote{overintegrated}, the number of quadrature nodes is $(p+1)+\text{overintegration}$.

For the curvilinear grid, it is important to use a non-symmetric mapping in every direction to ensure that none of the nonlinear metric terms implicitly factor off, as discussed in Cicchino~\textit{et al}.~\cite{cicchino2022provably}. The three-dimensional mapping is

\begin{equation}
   \begin{split}
   & [x,y,z]\in[x_L,x_R]^3,\:\:[a,b,c]\in[x_L,x_R]^3,\:\:l=\frac{x_R-x_L}{2\pi},\\
     &   x= a + \beta\sin{\left(\frac{a}{l}\right)}\sin{\left(\frac{b}{l}\right)}\sin{\left(\frac{2c}{l}\right)},\\
     &   y= b + \beta\sin{\left(\frac{4a}{l}\right)}\sin{\left(\frac{b}{l}\right)}\sin{\left(\frac{3c}{l}\right)},\\
      &  z = c +\beta\sin{\left(\frac{2a}{l}\right)}\sin{\left(\frac{5b}{l}\right)}\sin{\left(\frac{c}{l}\right)}.
    \end{split}\label{eq: warped grid}
\end{equation}

\noindent A cross-section for Eq.~(\ref{eq: warped grid}) is presented in Figure~\ref{fig: curv grid}.

\begin{figure}[H]
    \centering
    \includegraphics[width=0.4\textwidth]{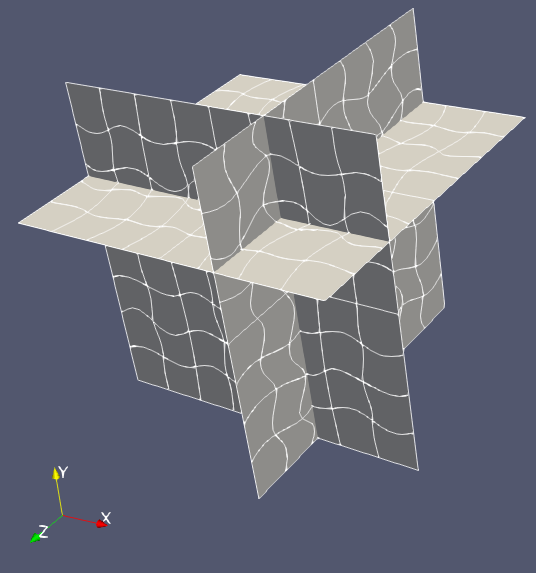}    
    \caption{3D Nonsymmetrically Warped Curvilinear Grid Cross-Section.}\label{fig: curv grid}
\end{figure}

\subsection{Manufactured Solution Convergence}\label{sec: manfac sol}

In this subsection we consider the manufactured solution for Euler's equations from Gassner \textit{et al}.~\cite{gassner2016split},

\begin{equation}
    \begin{split}
        &\rho = 2+\frac{1}{10}\sin{\left(\pi\left(x+y+z-2t\right)\right)},\\
        & u=1,\\
        &v=1,\\
       & w=1,\\
        & \rho e = \rho ^2,
    \end{split}
\end{equation}

\noindent with the unsteady source term,

\begin{equation}
    \begin{split}
    \bm{q} =   \begin{bmatrix}& c_1\cos{\left(\pi\left(x+y+z - 2t\right)\right)}\\
        & c_2\cos{\left(\pi\left(x+y+z - 2t\right)\right)}
                        + c_3 \sin{\left(2\pi\left(x+y+z - 2t\right)\right)}\\
        &  c_2\cos{\left(\pi\left(x+y+z - 2t\right)\right)}
                        + c_3 \sin{\left(2\pi\left(x+y+z - 2t\right)\right)}\\
        & c_2\cos{\left(\pi\left(x+y+z - 2t\right)\right)}
                        + c_3 \sin{\left(2\pi\left(x+y+z - 2t\right)\right)}\\
        &  c_4\cos{\left(\pi\left(x+y+z - 2t\right)\right)}
                        + c_5 \sin{\left(2\pi\left(x+y+z - 2t\right)\right)}
                        \end{bmatrix},
    \end{split}
\end{equation}

\noindent with $c_1 = \frac{\pi}{10}$, $c_2 = -\frac{\pi}{5}+\frac{\pi}{20}\left(1+5\gamma\right)$, $c_3=\frac{\pi}{100}, \left(\gamma-1\right)$, $c_4 = \frac{\pi}{20}\left(-7+15\gamma\right)$, and $c_5=\frac{\pi}{100}\left(3\gamma-2\right)$. The domain is $[x,y,z]\in[-1,1]^3$, and we use $\beta=\frac{1}{50}$ and $l=2$ in the warping Eq.~(\ref{eq: warped grid}). The $\Delta x$ is taken as the average distance between two quadrature points, $\frac{l}{M(p+1)}$, where $M$ is the total number of elements. We simulate for one cycle to a final time of $t_f=2$s. For the surface numerical flux, we add Roe dissipation~\cite{roe1981approximate} to the entropy conserving baseline flux. We integrate on Gauss-Legendre quadrature nodes with an uncollocated Lagrange modal basis. We used the weight-adjusted mass inverse approximation from Eq.~(\ref{eq: weight-adj ESFR mass}) with an adaptive timestep of $\Delta t =0.2\Delta x$.

The L2-error is computed as,

\begin{equation}
    \text{L}2 -\text{error} = \sqrt{\sum_{m=1}^{M}{\int_{\bm{\Omega}_m}{\left(u_m-u\right)^2 d\bm{\Omega}_m}}}
    =\sqrt{\sum_{m=1}^{M}\left(\bm{u}_m^T-\bm{u}_{exact}^T\right)\bm{W}\bm{J}_m\left(\bm{u}_m-\bm{u}_{exact}\right)}.
\end{equation}

The orders are presented in Fig.~\ref{fig: OOA man sol} for both density and pressure for the NSFR entropy conserving (EC) scheme with $c_\text{DG}$ and $c_+$.

\begin{figure}[H]
    \centering
   \begin{subfigure}[t]{0.5\textwidth}
        \centering
        \includegraphics[angle=270,width=1.0\textwidth]{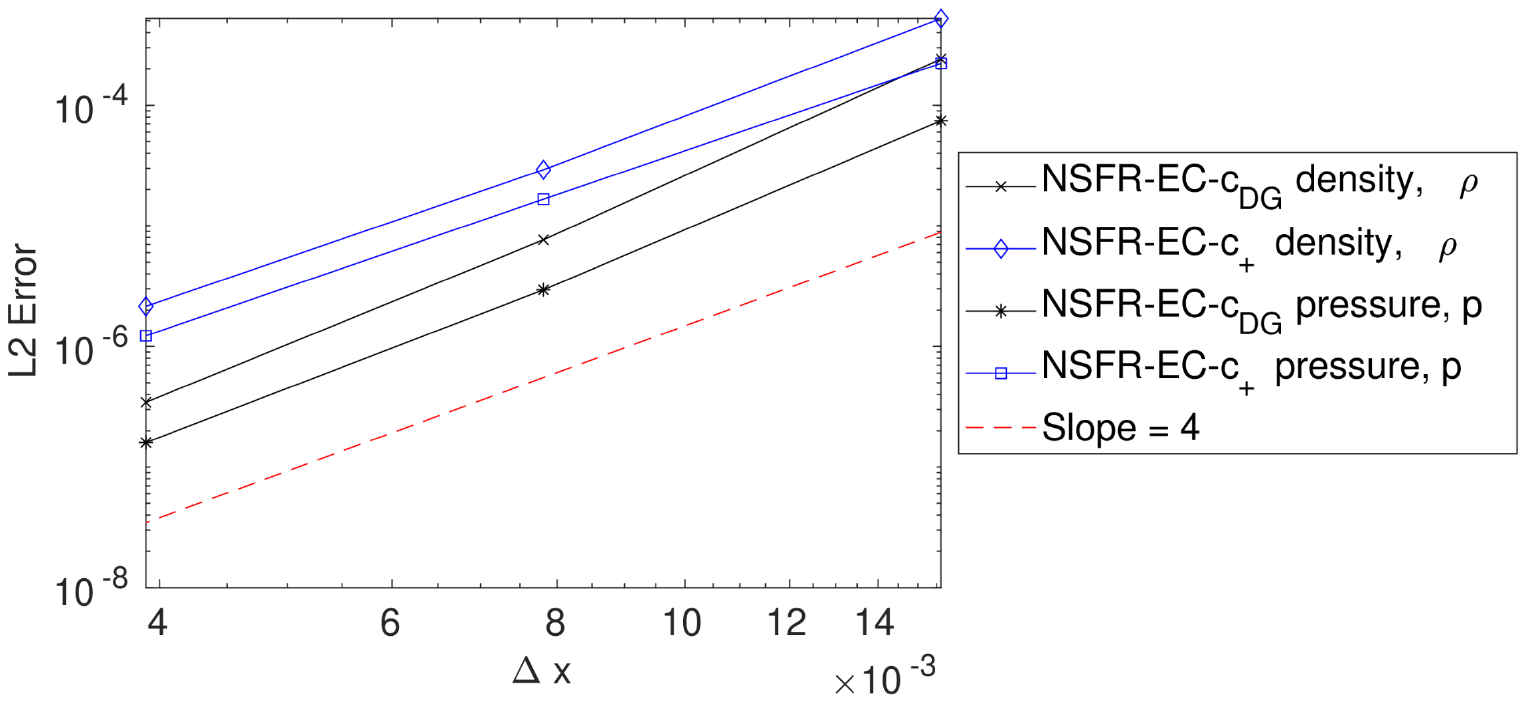}
    \caption{3D Manufactured Solution $p=3$ Orders of Convergence.}\label{fig: p3 ooa}
    \end{subfigure}%
    ~ 
   \begin{subfigure}[t]{0.5\textwidth}
        \centering
       \includegraphics[angle=270,width=1.0\textwidth]{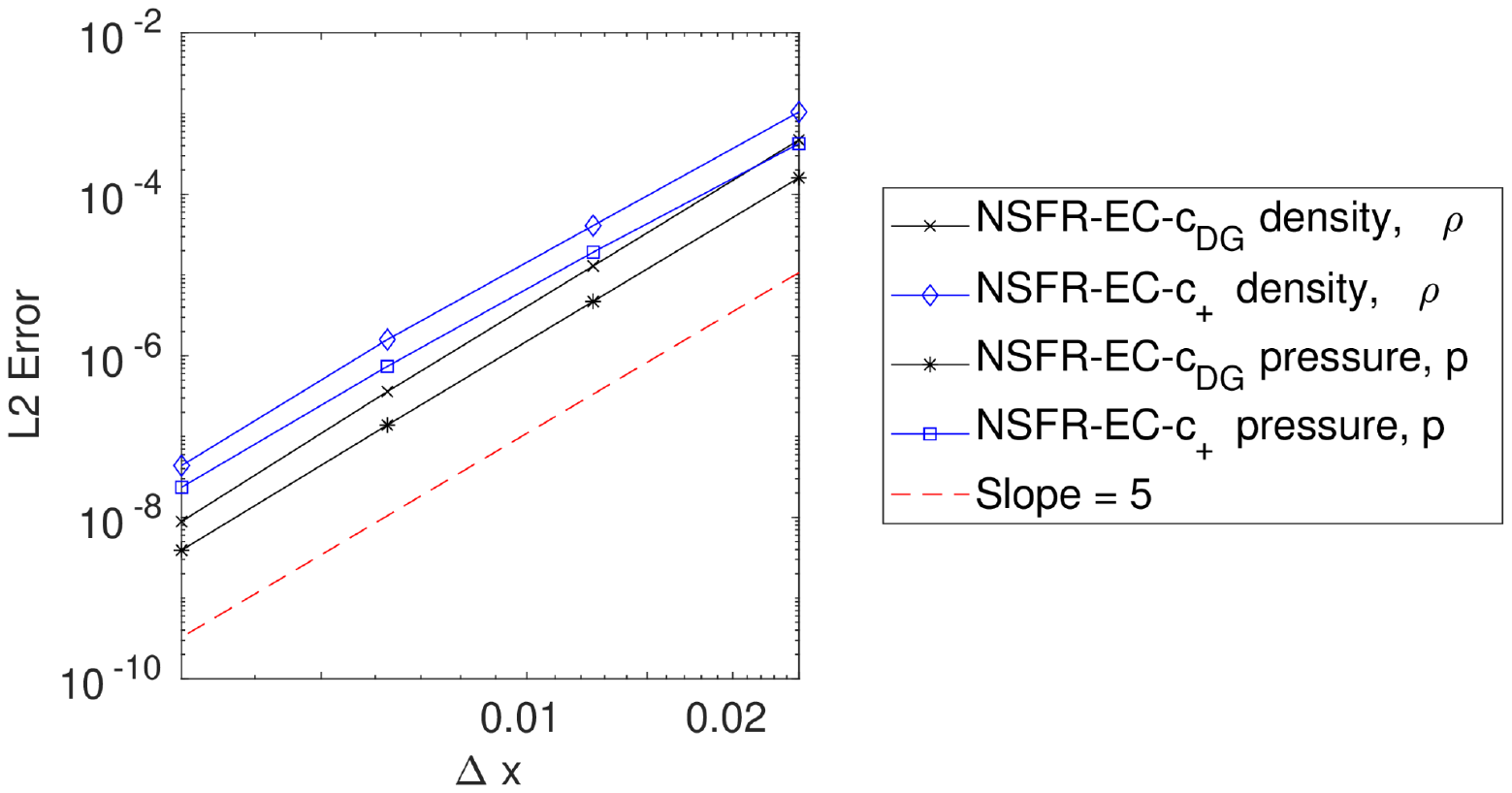}
    \caption{3D Manufactured Solution $p=4$ Orders of Convergence.}\label{fig: p4 ooa}
    \end{subfigure}
    \begin{subfigure}[t]{0.5\textwidth}
        \centering
       \includegraphics[angle=270,width=1.0\textwidth]{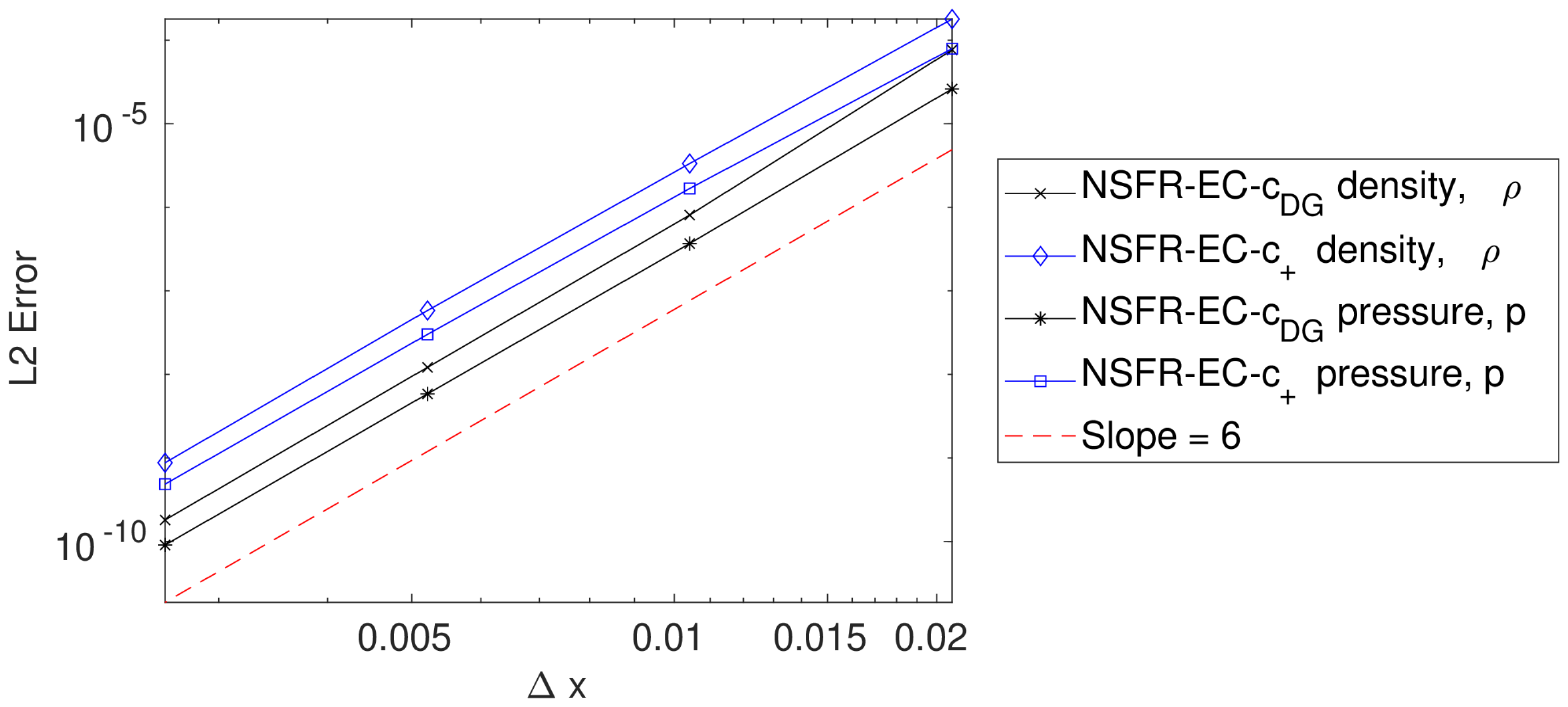}
    \caption{3D Manufactured Solution $p=5$ Orders of Convergence.}\label{fig: p5 ooa}
    \end{subfigure}
    \caption{3D Manufactured Solution Orders of Convergence}\label{fig: OOA man sol}
\end{figure}

From Fig.~\ref{fig: OOA man sol}, we observe that the NSFR discretization converges at $p+1$ for both $c_\text{DG}$ and $c_+$ variants in curvilinear coordinates. The error levels obtained with $c_+$ were slightly larger than those obtained with $c_\text{DG}$ for the same grid level and polynomial order as expected. The larger error levels are a trade-off for larger maximum timesteps offered by $c_+$.

\subsection{Maximum CFL}\label{sec: max cfl}

In this subsection we compare the maximum CFL numerically obtained for different NSFR schemes. We consider the manufactured solution from Sec.~\ref{sec: manfac sol} and run it for different NSFR correction parameter choices, different integration nodes, and different polynomial degrees. The manufactured solution is integrated until a final time $t_f=1.0$s for each combination of parameters with a CFL$=0.1$. Then, we increase the CFL by $0.01$ and rerun the test until the computed pressure error at $t_f=1.0$s exceeds seven significant digits as compared to the solution obtained with a CFL$=0.1$. GL implies that the scheme was integrated on Gauss-Legendre quadrature nodes, and LGL signifies that the scheme was integrated on Gauss-Legendre-Lobatto quadrature nodes. 

\begin{table}[H]
\subfloat[Max CFL for $p=3$, $4^3$ elements.]{
 \begin{tabular}{c c c}
{\bf Scheme} & {\bf Quadrature} &  {\bf Max CFL}\\ \hline 
NSFR $c_\text{DG}$& GL & 0.18\\ 
 & &\\\hline 
NSFR $c_+$ & GL &   0.21\\ 
  & &\\\hline 
  NSFR $c_\text{HU}$ & GL&  0.22\\ 
 & &\\\hline 
NSFR $c_\text{DG}$ &LGL & 0.16 \\ 
  & &\\\hline 
NSFR $c_\text{+}$ &LGL& 0.24\\ 
 & &\\\hline 
\end{tabular} 
 }
{
\subfloat[Max CFL for $p=4$, $4^3$ elements.]{
 \begin{tabular}{c c c}
{\bf Scheme} & {\bf Quadrature} &  {\bf Max CFL}\\ \hline 
NSFR $c_\text{DG}$& GL & 0.17\\ 
 & &\\\hline 
NSFR $c_+$ & GL  & 0.29\\ 
  & &\\\hline 
  NSFR $c_\text{HU}$ & GL& 0.24\\ 
 & &\\\hline 
NSFR $c_\text{DG}$ & LGL& 0.14\\ 
  & &\\\hline 
NSFR $c_\text{+}$ & LGL &0.28\\ 
 & &\\\hline 
\end{tabular} 
 }
 }
\quad
{
\subfloat[Max CFL for $p=5$, $4^3$ elements.]{
 \begin{tabular}{c c c}
{\bf Scheme} & {\bf Quadrature} &  {\bf Max CFL}\\ \hline 
NSFR $c_\text{DG}$& GL& 0.17\\ 
 & \\\hline 
NSFR $c_+$ & GL & 0.22\\ 
 & &\\\hline 
NSFR $c_\text{HU}$ & GL& 0.23\\ 
 & &\\\hline 
NSFR $c_\text{DG}$ & LGL & 0.16\\ 
 && \\\hline 
NSFR $c_\text{+}$ & LGL & 0.20\\ 
 & &\\\hline 
\end{tabular} 
 }
 }
 \caption{Max CFL for $p=4$ and $p=5$, $4^3$ elements.} \label{tab: curvilinear CFL}
 \end{table}

Table~\ref{tab: curvilinear CFL} shows an increase in CFL from $c_\text{DG}$ to $c_+$ as that found for linear advection by Vincent~\textit{et al}.~\cite{vincent_insights_2011}, but the difference in CFL values is not as large as the linear case from Vincent~\textit{et al}.~\cite{vincent_insights_2011}. In the NSFR implementation Eq.~(\ref{eq: NSFR}), this increase in CFL comes without additional runtime computational cost. Unlike the results from Gassner and Kopriva~\cite{gassner2011comparison}, we did not find an increase in CFL by using LGL nodes as compared to GL nodes, except only for $p=3$ and $c_+$. We believe this was due to the loss of integration strength with LGL nodes as compared to GL nodes. Gassner and Kopriva~\cite{gassner2011comparison} considered a linear flux on a linear grid, where LGL nodes would integrate the volume divergence of the flux exactly, whereas, for our cases, the flux is rational on a curvilinear mesh. For the nonlinear case, to replicate the filter obtained in Gassner and Kopriva~\cite{gassner2011comparison} we also ran a value of $c_\text{HU}$. When integrated on GL nodes with a value of $c_\text{HU}$ and a Lagrange basis with LGL solution nodes, the modified mass matrix is exactly the LGL collocated mass matrix~\cite{huynh_flux_2007,de_grazia_connections_2014,vincent_new_2011}. Thus, the nonlinear extension for the CFL increase from Gassner and Kopriva~\cite{gassner2011comparison} would be using NSFR with $c_\text{HU}$. 

For $p=3$ and $p=5$, $c_\text{HU}$ gave a slightly larger CFL than $c_+$. The value of $c_+$ used was numerically obtained from the von Neumann analysis for linear advection by Vincent and coauthors~\cite{vincent_insights_2011}. This case considers a rational flux, and due to the nonlinearities, the $c$ value that corresponds to the maximum timestep might be slightly lower than the $c_+$ value for one-dimensional linear advection. A full conclusion on the value of $c_+$ for nonlinear and rational fluxes cannot be drawn until a von Neumann analysis for nonlinear problems is performed for the cases. 

\subsection{Inviscid Taylor Green Vortex}\label{sec: TGV}

We consider the inviscid Taylor-Green vortex problem, initialized on the periodic box as,

\begin{equation}
    \begin{split}
       & \rho=1,\\
       & u=\sin{x}\cos{y}\cos{z},\\
       & v=-\cos{x}\sin{y}\cos{z},\\
       & w=0,\\
      &  p=\frac{100}{\gamma}+\frac{1}{16}\Big( \cos{2x}\cos{2z} +2\cos{2x}+2\cos{2y}+\cos{2y}\cos{2z}
        \Big),\\
       & \bm{x}^c\in[0,2\pi]^3,\: t\in[0,14].
    \end{split}
\end{equation}

We consider both a Cartesian mesh and the heavily warped, periodic grid defined by Eq.~(\ref{eq: warped grid}) with $\beta=\frac{1}{5}$. 
We use $4$ elements in each direction and verify that the discrete change in entropy, $\hat{\bm{v}}\left(\bm{M}_m+\bm{K}_m \right)\frac{d}{dt}\hat{\bm{u}}(t)^T$, is conserved on the order of $1\times 10^{-13}$ for $p=4,5$. These polynomial orders and mesh size correspond to $20^3$ and $24^3$ degrees of freedom to simulate under-resolved turbulence. 
This test showcases the strength of the entropy conserving framework because a conservative DG scheme, without the entropy stable framework, diverges. All schemes were globally conservative and free-stream preserving on the order of $1\times 10^{-16}$. In Fig.~\ref{fig: TGV change entropy}, for a polynomial degree of $4$, we plot the discrete change in entropy for $c_\text{HU}$ and $c_+$ with the weight-adjusted inverse, and $c_+$ without using a weight-adjusted inverse on the curvilinear grid to demonstrate the discrete entropy conservation. 


 \begin{table}[H]
\centering{
 \begin{tabular}{c c c c }
Scheme &  Overintegration & Discrete Entropy Conserved $\mathcal{O}$(1e-13)\\ \hline 
Cons. DG-GL & 0 & No \\ \hline 
NSFR-EC-LGL $c_\text{DG}$ & 0 & Yes \\ \hline 
NSFR-EC-LGL $c_\text{DG}$ & 3 & Yes \\ \hline 
NSFR-EC-GL $c_\text{DG}$ & 0 & Yes \\ \hline 
NSFR-EC-GL $c_\text{DG}$ & 3 & Yes \\ \hline
NSFR-EC-LGL $c_{+}$ & 0 & Yes \\ \hline 
NSFR-EC-LGL $c_{+}$ & 3 & Yes \\ \hline 
NSFR-EC-GL $c_{+}$ & 0 & Yes \\ \hline 
NSFR-EC-GL $c_{+}$ & 3 & Yes \\ \hline 
\end{tabular} 
 }
 \caption{Change in Entropy Results $p=4$,$5$ Cartesian Mesh.} \label{tab: cartesian entropy}
 \end{table} 

 \begin{table}[H]
\centering
{
 \begin{tabular}{c c c c }
Scheme &  Overintegration & Discrete Entropy Conserved $\mathcal{O}$(1e-13)\\ \hline 
Cons. DG-GL & 0 & No \\ \hline 
NSFR-EC-LGL $c_\text{DG}$ & 0 & Yes \\ \hline 
NSFR-EC-LGL $c_\text{DG}$ & 3 & Yes \\ \hline 
NSFR-EC-GL $c_\text{DG}$ & 0 & Yes \\ \hline 
NSFR-EC-GL $c_\text{DG}$ & 3 & Yes \\ \hline 
NSFR-EC-LGL $c_{+}$ & 0 & Yes \\ \hline 
NSFR-EC-LGL $c_{+}$ & 3 & Yes \\ \hline 
NSFR-EC-GL $c_{+}$ & 0 & Yes \\ \hline 
NSFR-EC-GL $c_{+}$ & 3 & Yes \\ \hline 
\end{tabular} 
 }
 \caption{Change in Entropy Results $p=4$,$5$ Curvilinear Mesh.} \label{tab: curvilinear entropy}
 \end{table}

\begin{figure}[H]
    \centering
    \includegraphics[angle=270,width=0.8\textwidth]{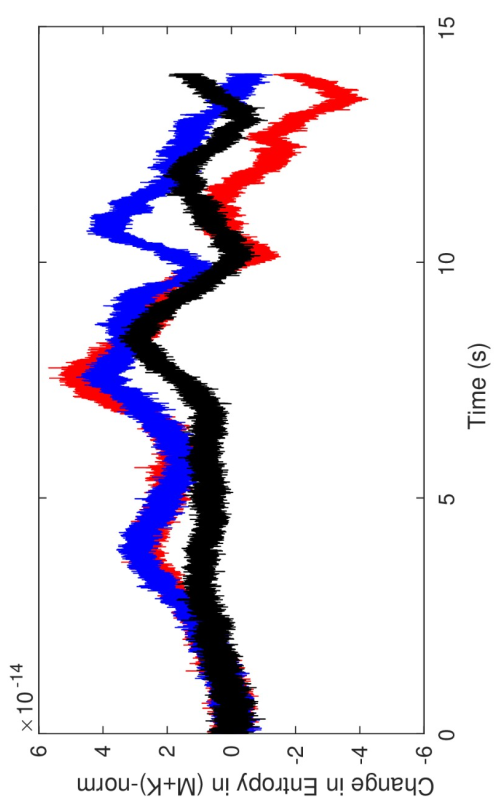}
    \caption{Change in Entropy. Red line $c_\text{HU}$ with weight-adjusted, black line $c_+$ with weight adjusted, and blue line $c_+$ without weight-adjusted mass matrix.}
    \label{fig: TGV change entropy}
\end{figure}

As we can see from Tables~\ref{tab: cartesian entropy} and~\ref{tab: curvilinear entropy}, for both a Cartesian and curvilinear mesh, the weight-adjusted NSFR entropy conserving discretization discretely conserves the change in entropy on the order of 1e-13 for arbitrary quadrature integration---provided the integration strength is exact for at least $2p-1$ polynomials. The imperative of showing varying quadrature rules is that, as shown by Winter~\textit{et al}.~\cite{winters2018comparative}, the integration strength affects the polynomial dealiasing for turbulent flows. Since the application for NSFR will be in predicting turbulent flows, Tables~\ref{tab: cartesian entropy} and~\ref{tab: curvilinear entropy} demonstrate the promising flexibility that the scheme offers.

We additionally plot the change in kinetic energy for the LGL and GL DG cases in Fig.~\ref{fig:KEwithoutPressure} to verify Lemma~\ref{lem: kin energy}. We first verify that the kinetic energy conserving flux discretely satisfies 
$$\sum_{n=1}^{n_\text{state}} \hat{\bm{v}}_\text{K.E.,n}\bm{\chi}\left(\bm{\xi}_v^r\right)^T\left[\left(\bm{W}\nabla^r\bm{\phi}\left(\bm{\xi}_v^r\right)-\nabla^r\bm{\phi}\left(\bm{\xi}_v^r\right)^T\bm{W}\right)\circ \left(\Tilde{\bm{F}}_{m,v}^r-\Tilde{\bm{P}}_{m,v}^r\right)\right]\bm{1}^T =\sum_{k=1}^{d}\psi^k_\text{K.E.}=0 $$ 
on the volume quadrature nodes. 
In Fig.~\ref{fig:KEwithoutPressure}, this is demonstrated to machine precision for both GL and LGL nodes. Here $\Tilde{\bm{F}}_{m,v}^r$ is the two-point flux for only the volume nodes, $$\left(\Tilde{\bm{F}}_{m,v}^r\right)_{ij} = {\bm{f}_s\left(\Tilde{\bm{u}}_m(\bm{\xi}_{i}^r),\Tilde{\bm{u}}_m(\bm{\xi}_{j}^r)\right)}
     {\left(
       \frac{1}{2}\left(\bm{C}_m(\bm{\xi}_{i}^r)+ \bm{C}_m(\bm{\xi}_{j}^r) \right)\right)},\:\forall\: 1\leq i,j\leq N_v,$$ and $\Tilde{\bm{P}}_{m,v}^r$ is the volume pressure work, $$\left(\Tilde{\bm{P}}_{m,v}^r\right)_{ij} = \frac{1}{2}\left(p_i+p_j\right)\bm{I}_d{\left(
       \frac{1}{2}\left(\bm{C}_m(\bm{\xi}_{i}^r)+ \bm{C}_m(\bm{\xi}_{j}^r) \right)\right)},\:\forall\: 1\leq i,j\leq N_v$$ with $\bm{I}_d$ the $d\times d$ identity matrix.

\begin{figure}[H]
    \centering
   \begin{subfigure}[t]{0.5\textwidth}
        \centering
        \includegraphics[angle=270,width=1.0\textwidth]{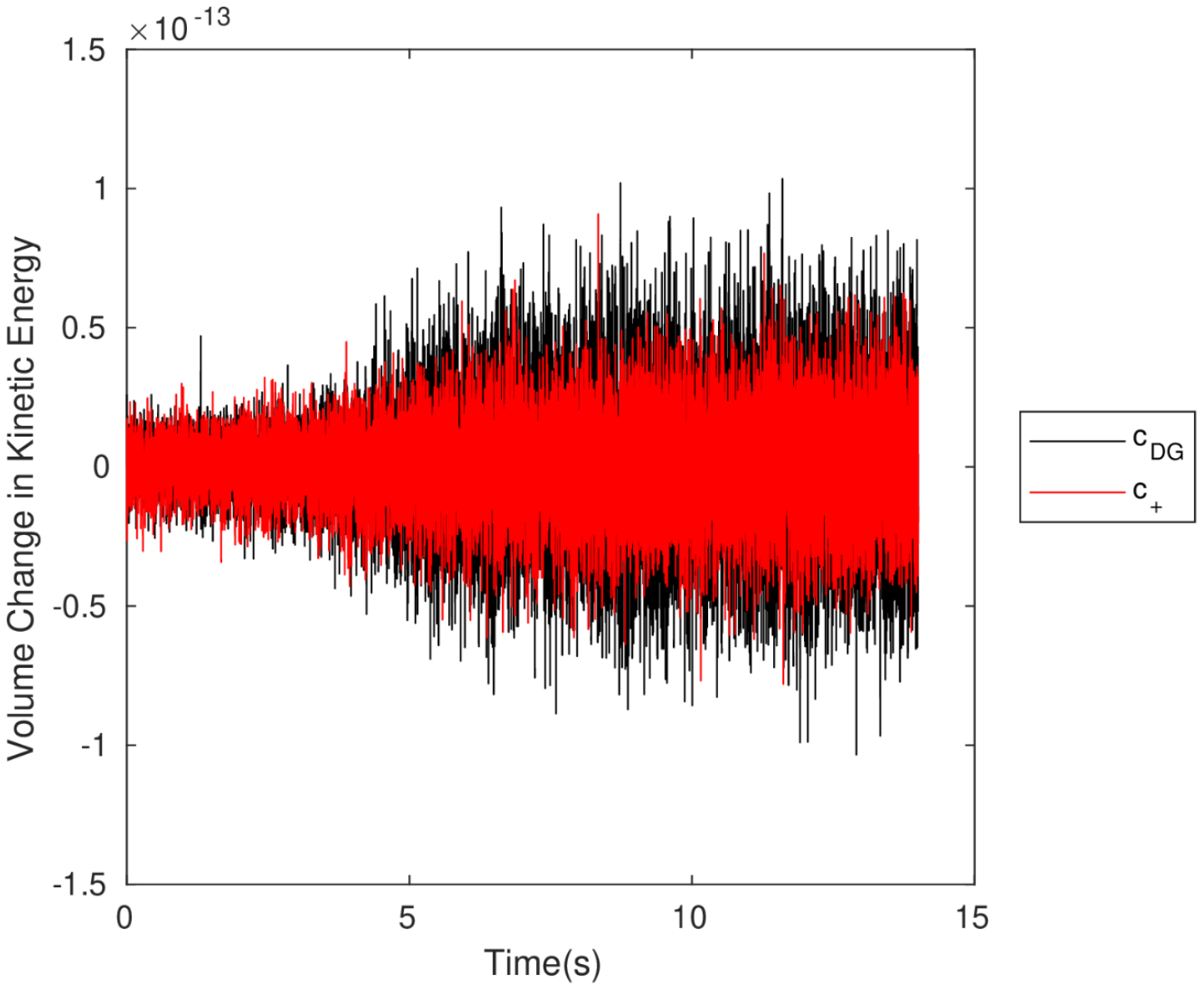}
    \caption{Volume Change in Kinetic Energy without Pressure Work LGL.}\label{subfig: TGV vol change energy LGL}
    \end{subfigure}%
    ~ 
   \begin{subfigure}[t]{0.5\textwidth}
        \centering
       \includegraphics[angle=270,width=1.0\textwidth]{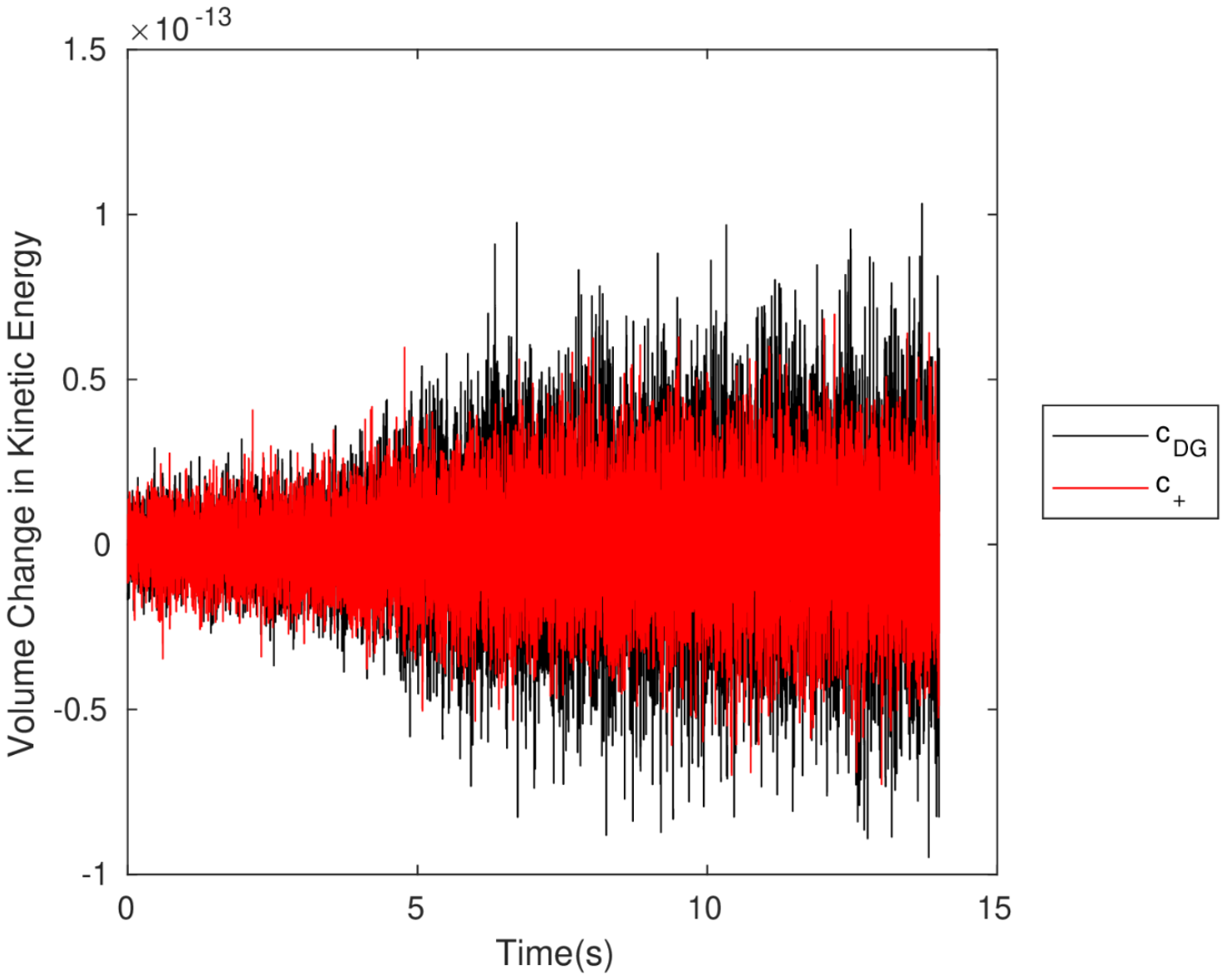}
    \caption{Volume Change in Kinetic Energy without Pressure Work GL.}\label{subfig: TGV vol change energy GL}
    \end{subfigure}
 \caption{Volume Change in Kinetic Energy without Pressure Work for LGL and GL Nodes.}
 \label{fig:KEwithoutPressure}
\end{figure}

\begin{figure}[H]
    \centering
   \begin{subfigure}[t]{0.5\textwidth}
        \centering
        \includegraphics[angle=270,width=1.0\textwidth]{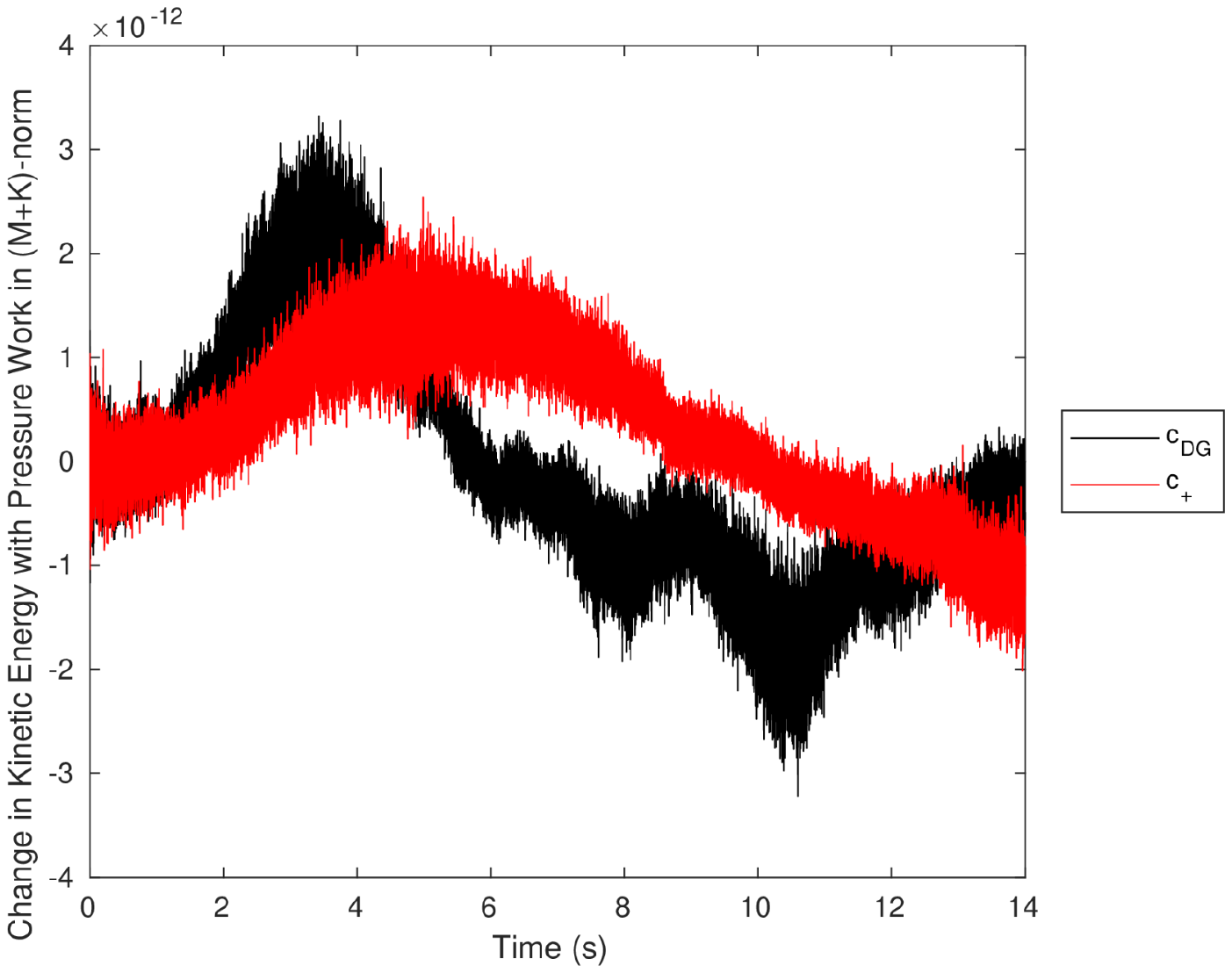}
    \caption{Change in Kinetic Energy without Pressure Work LGL.}\label{subfig: TGV change energy LGL total}
    \end{subfigure}%
    ~ 
   \begin{subfigure}[t]{0.5\textwidth}
        \centering
       \includegraphics[angle=270,width=1.0\textwidth]{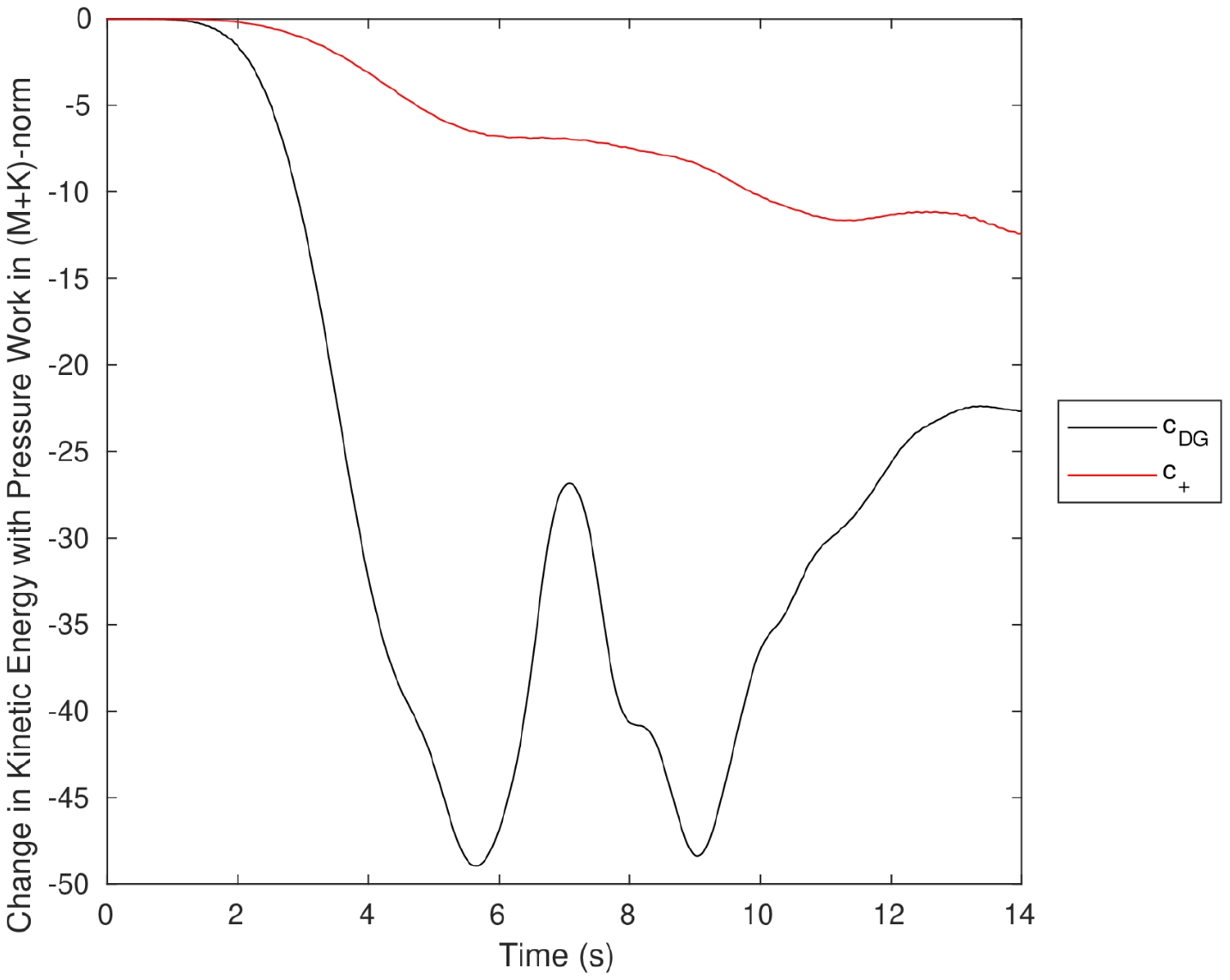}
    \caption{Change in Kinetic Energy without Pressure Work GL.}\label{subfig: TGV change energy GL total}
    \end{subfigure}
 \caption{Change in Kinetic Energy without Pressure Work.}
\end{figure}

Next, we compute the total change in kinetic energy by $\hat{\bm{v}}_\text{K.E.}\left(\bm{M}_m+\bm{K}_m \right)\frac{d}{dt}\hat{\bm{u}}(t)^T-P_\text{work}$, where $P_\text{work}$ is the global pressure work by integrating the volume and surface terms. From Fig.~\ref{subfig: TGV change energy LGL total}, we see for LGL nodes, the total change in kinetic energy is conserved on the order of 1e-12 since it does not require the inverse mapping of kinetic energy variables to the surface, whereas in Fig.~\ref{subfig: TGV change energy GL total}, for GL nodes, it is not conserved. Although a kinetic energy conserving numerical flux was used, from Lemma~\ref{lem: kin energy}, the LGL case discretely conserved kinetic energy to machine precision, whereas the GL case did not. This verifies Lemma~\ref{lem: kin energy} that global kinetic energy cannot be conserved when the surface nodes are not a subset of the volume nodes because the inverse mapping from $\bm{v}_\text{K.E.}\to\bm{u}$ does not exit.

\subsection{NSFR versus DG Conservative with and without Overintegration}\label{sec: comparison}

Next, using our proposed sum-factorized Hadamard product from Cicchino and Nadarajah~\cite{CicchinoHadamardScaling2023}, we wish to compare the performance of the entropy conserving scheme with the conservative DG scheme using sum-factorization techniques. We solve the three-dimensional inviscid TGV problem, with Gauss-Legendre quadrature nodes on the curvilinear grid defined in Eq.~(\ref{eq: warped grid}), with $4^3$ elements and $\beta=\frac{1}{5}$. 
We solve it in six different ways. First, with the conservative DG scheme that does not require a Hadamard product in Eq.~(\ref{eq: cons DG strong}). Second, the conservative DG scheme overintegrated by $2(p+1)$ to resemble exact integration for a cubic polynomial on a curvilinear grid. We consider overintegration because it is another tool used for stabilization~\cite{winters2018comparative} through polynomial dealiazing. Lastly, with our NSFR entropy conserving scheme that requires an uncollocated Hadamard product along with entropy projection techniques. We then, in dashed lines, run the same tests with an FR correction value of $c_{+}$~\cite{vincent_insights_2011} to compare the additional cost of FR versus its DG equivalent. For the test, we perform 10 residual solves sequentially and record the total CPU time. The test was run in parallel with 4 Intel i5-8600 CPUs with 4GB of DDR4 RAM. The CPU time presented is the sum across ranks of the CPU times on each processor.

 \begin{figure}[H]
    \centering
    \includegraphics[angle=270,width=0.7\textwidth]{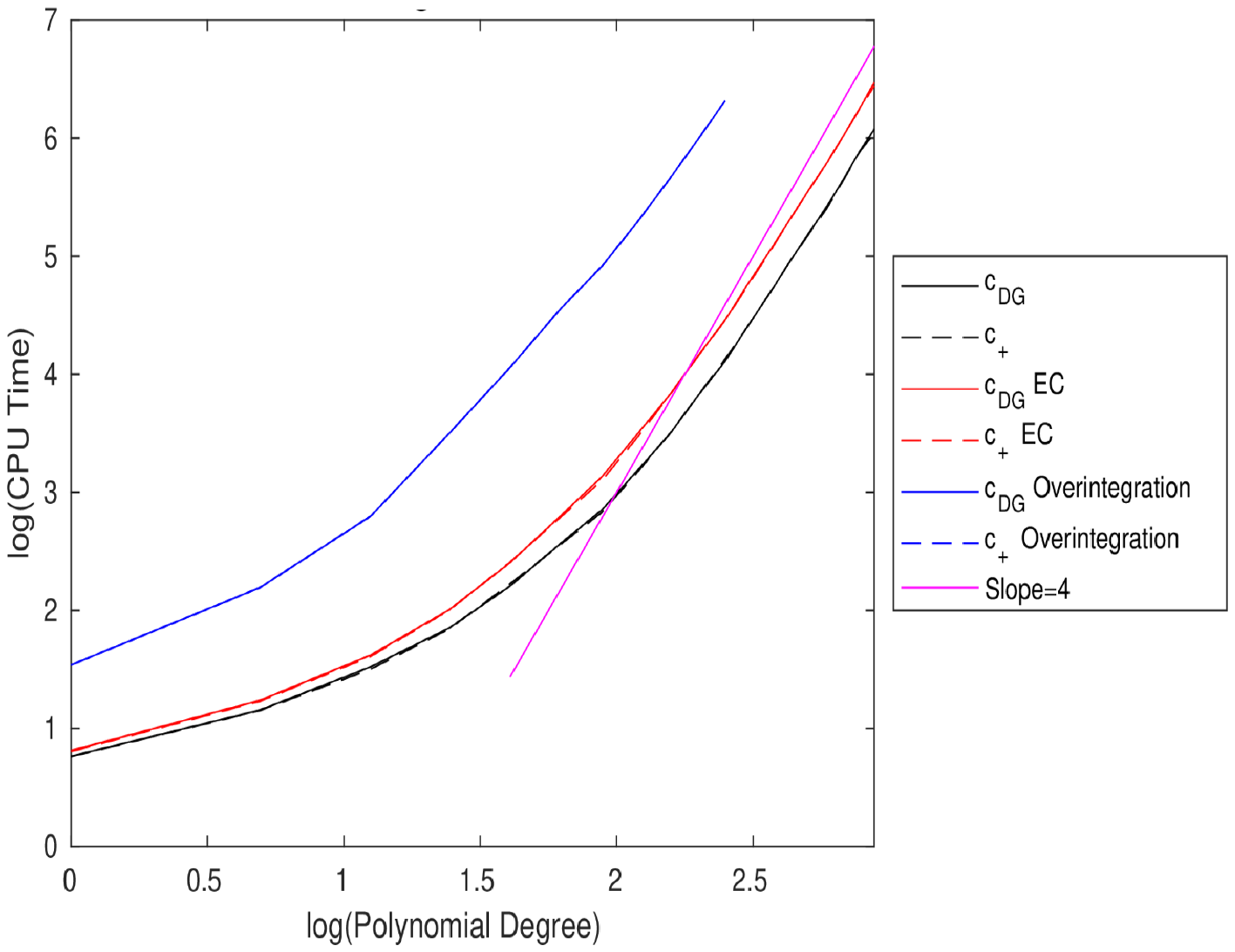}
    \caption{CPU time versus polynomial degree TGV.}\label{fig: Philip scaling}\label{fig: tgv scale}
\end{figure}

From Figure~\ref{fig: tgv scale}, all three methods have the solver scale at order $\mathcal{O}\left(p^{d+1}\right)$ in curvilinear coordinates because they exploit sum-factorization~\cite{orszag1979spectral} for the matrix-vector products, and the NSFR-EC scheme uses our proposed sum-factorized Hadamard product evaluation~\cite{CicchinoHadamardScaling2023}. The blue line representing the overintegrated conservative DG scheme took the most amount of CPU time and was run until $p=20$ due to memory limitations. 
The conservative DG and NSFR-EC schemes took a comparable amount of CPU time in Fig.~\ref{fig: tgv scale}. This result is dependent on the sum-factorized Hadamard product from Cicchino and Nadarajah~\cite{CicchinoHadamardScaling2023}---the sum-factorized Hadamard product is evaluated in the same number of flops as the DG divergence of the flux with sum-factorization for the matrix-vector product. The computational cost difference between the DG conservative scheme and the NSFR-EC $c_{\text{DG}}$ is in the evaluation of the two-point flux as compared to the convective flux at a single quadrature node. 
Also, there was a negligible computational cost difference between all $c_\text{DG}$ versus $c_+$ schemes since the mass matrix inverse was approximated in a weight-adjusted form. From Fig.~\ref{fig: tgv scale}, it appears that using the algorithm in Cicchino and Nadarajah~\cite[Sec. 2]{CicchinoHadamardScaling2023}, the entropy conserving scheme's cost is more comparable to the conservative DG scheme rather than an overintegrated/exactly integrated DG scheme.

To further demonstrate the performance differences between the NSFR-EC-$c_\text{DG}$ scheme using the \enquote{sum-factorized} Hadamard product evaluations detailed in Cicchino and Nadarajah~\cite[Sec. 2]{CicchinoHadamardScaling2023} and the conservative DG scheme in curvilinear coordinates, we run the inviscid TGV on a non-symmetrically warped curvilinear grid and compare the wall clock times. All schemes use an uncollocated, modal Lagrange basis, and are integrated on Gauss-Legendre quadrature nodes. The conservative DG scheme overintegrated by $2(p+1)$ to resemble exact integration for a cubic polynomial on a curvilinear grid. We integrate in time with a 4-$th$ order Runge-Kutta time-stepping scheme with an adaptive Courant-Friedrichs-Lewy value of 0.1 until a final time of $t_f = 14$ s. For NSFR-EC $c_\text{DG}$  we use Chandrashekar's flux~\cite{chandrashekar2013kinetic} in the volume and surface with Ranocha's pressure fix~\cite{ranocha2022preventing}. For the DG conservative scheme, we use the Roe~\cite{roe1981approximate} surface numerical flux. All of the tests were run on 1 node with 8 AMD Rome 7532 processors with 4GB of RAM on each CPU provided by the Digital-Alliance of Canada. Each test was run 4 times and we present the average of the 4-tests' max wall clock time for a single processor and the average of the 4-tests' 
total CPU time across the 8 processors.

\begin{table}[H]\centering
\begin{tabular}{|c|c|c|c|c|}
\hline
$p$ & Number of Elements & Scheme & Max Wall Clock (hours) & Total CPU Time (hours)\\
&&& & \\
\hline
3& $4^3$ &NSFR-EC-$c_\text{DG}$ &0.1111
&0.8713
\\
\cline{3-5}
& &DG-cons  & 0.09416 
&0.7376 
\\
\cline{3-5}
& &DG-cons-overint&0.5033
&3.816
\\
\cline{2-5}
& $8^3$ &NSFR-EC-$c_\text{DG}$ & 1.530
&12.06
\\
\cline{3-5}
& &DG-cons&1.520
&11.64
\\
\cline{3-5}
& &DG-cons-overint&5.740
&44.42
\\
\hline
4& $4^3$&NSFR-EC-$c_\text{DG}$ &0.2756
&2.170
\\
\cline{3-5}
& &DG-cons&0.2669
&2.053
\\
\cline{3-5}
& &DG-cons-overint& 0.9977
&7.767
\\
\cline{2-5}
& $8^3$&NSFR-EC-$c_\text{DG}$ &3.873
&30.55
\\
\cline{3-5}
& &DG-cons& 3.115
& 24.59
\\
\cline{3-5}
& &DG-cons-overint&12.68
& 100.3 
\\
\hline
5&  $4^3$&NSFR-EC-$c_\text{DG}$ &0.5964
&4.470
\\
\cline{3-5}
& &DG-cons& Crashed & Crashed\\
\cline{3-5}
& &DG-cons-overint&2.013
&15.10
\\
\cline{2-5}
&  $8^3$&NSFR-EC-$c_\text{DG}$ &7.052
&52.58
\\
\cline{3-5}
& &DG-cons&Crashed & Crashed \\
\cline{3-5}
& &DG-cons-overint& 23.41
&185.2
\\
\hline
\end{tabular}\caption{TGV Comparison for CPU and Wall Clock Times.}\label{tab: wall clock table}
\end{table}

From Table~\ref{tab: wall clock table}, averaging all of the tests, the NSFR-EC-$c_\text{DG}$ scheme took about a 13\% longer CPU time and 12\% longer wall clock time as compared to conservative DG. On average, the overintegrated conservative DG scheme took 364\% longer CPU time and 365\% longer wall clock time than the NSFR-EC-$c_\text{DG}$ scheme. 
This small percentage difference between DG conservative and NSFR-EC-$c_\text{DG}$ demonstrates how the sum-factorized Hadamard product algorithm in Cicchino and Nadarajah~\cite{CicchinoHadamardScaling2023} has drastically reduced the computational cost of computing a two-point flux. 
The $p=5$ DG conservative scheme diverged at $t=9.06$s on the $4^3$ mesh, and at $t=6.70$s on the $8^3$ mesh. This further demonstrates the advantage of the NSFR-EC scheme since it has provable guaranteed nonlinear stability for a reasonable computational cost trade-off. 
From both Fig.~\ref{fig: Philip scaling} and Table~\ref{tab: wall clock table}, it is clear that with the proposed sum-factorized Hadamard product, entropy conserving and stable methods are computationally competitive with classical DG schemes.

\section{Conclusion}

We demonstrated a novel, low-storage, weight-adjusted approach for NSFR schemes in curvilinear coordinates. 
{\color{black}In the context of an uncollocated case, both theoretical proof and numerical validation have demonstrated the attainability of discrete entropy conservation. Conversely, the preservation of discrete kinetic energy is achievable only for the collocated case on Gauss-Legendre-Lobatto quadrature nodes.} 

Additionally, in curvilinear coordinates, a unique, dense Gram matrix needs to be inverted within every element. This issue was circumvented by a weight-adjusted approach, where the operation cost was reduced to inverting a unique diagonal matrix storing the determinant of the metric Jacobian at each quadrature node. Coupled with sum-factorization, and the sum-factorized Hadamard product, our proposal greatly accelerates the run-time. For different FR schemes, the weight-adjusted framework maintained nonlinear stability for the inviscid Taylor-Green vortex problem on an extremely coarse, heavily warped curvilinear grid, making it an attractive implementation for FR in the high-performance computing context.

\section{Acknowledgements}

We would like to gratefully acknowledge the financial support of the Natural Sciences and Engineering Research Council of Canada (NSERC) Discovery Grant Program, NSERC Postgraduate Scholarships -- Doctoral program, and McGill University.

\bibliographystyle{model1-num-names}
\bibliography{bibliographie}

\end{document}